\documentclass[reqno]{amsart}
\usepackage{amsmath, amsthm, amssymb, amsfonts}
\usepackage{mathtools}
\usepackage{subcaption}
\usepackage[normalem]{ulem}
\usepackage[numbers]{natbib}
\usepackage{graphicx}
\usepackage{booktabs}
\usepackage{xcolor}
\usepackage{tikz}
\usepackage[foot]{amsaddr}
\usepackage{hyperref}
\usepackage{longtable}
\hypersetup{
    colorlinks=true,
    citecolor=green!50!black,
    linkcolor=red!80!black,
    filecolor=blue,
    urlcolor=blue!70!black,
}
\usepackage[normalem]{ulem}
\usepackage{enumitem}
\usepackage[margin=1.5in]{geometry}

\newtheorem{theorem}{Theorem}[section]
\newtheorem{proposition}[theorem]{Proposition}

\newtheorem{lemma}[theorem]{Lemma}
\theoremstyle{definition}
\newtheorem{remark}[theorem]{Remark}
\newtheorem{example}[theorem]{Example}
\newtheorem{definition}[theorem]{Definition}

\newcommand{\arxivonly}[1]{#1}
\newcommand{\journalonly}[1]{}

 \title[Madelung hydrodynamics and Poisson geometry of wave functions]{Madelung hydrodynamics \\and Poisson geometry of wave functions}

\author{Boris Khesin${^{1\journalonly{,*}}}$}
\address{${^1}$Department of Mathematics, University of Toronto, Canada}
\arxivonly{
\email{khesin@math.toronto.edu}
}
\journalonly{
\thanks{${}^*$Corresponding author, \texttt{khesin@math.toronto.edu}}
}

\author{Klas Modin${^{2}}$}
\address{${^2}$Department of Mathematical Sciences, Chalmers University of Technology and University of Gothenburg, SE-412~96 Gothenburg, Sweden}
\arxivonly{
\email{klas.modin@chalmers.se}
}


\date{}

\newcommand{\WF}{\Psi}
\newcommand{\WFnonzero}{\Psi_{\not=0}}

\newcommand*{\WFgamma}[1][\gamma]{\Psi_{#1}}

\newcommand{\WFunion}[1][\gamma]{\Psi_{\cup #1}}

\newcommand{\Dens}{\mathrm{Dens}}
\newcommand{\Densgamma}[1]{\Dens_{\gamma}}
\newcommand{\Densunion}[1]{\Dens_{\cup\gamma}}
\newcommand{\Mom}{\mathbf{M}}
\newcommand{\DC}{\mathcal{DC}}
\newcommand{\hd}{\boldsymbol{\psi}}

\begin{document}

\begin{abstract}
We describe the Poisson geometry of the Madelung transform between quantum mechanics and hydrodynamics for generic wave functions. We prove that for arbitrary oriented manifolds this transform, being regarded as a momentum map, naturally defines prequantum bundles for coadjoint orbits of semidirect extensions of diffeomorphism groups. Furthermore, we show that the Madelung framework provides a natural infinite-dimensional version of the convexity results for Hamiltonian torus actions, thus giving a partial answer to Atiyah's question. 

In particular, for wave functions without zeros our results provide a K\"ahler map between the infinite-dimensional Fubini--Study and Fisher--Rao geometries, thus extending previous results to non-simply-connected manifolds. Furthermore, for wave functions with noncritical zeros, the Madelung transform is shown to be a symplectomorphism to the coadjoint orbits with Morse--Bott densities. The latter, in turn, furnishes a novel momentum map point of view on the Wallstrom quantization condition for the hydrodynamical form of quantum mechanics. 
We also comment on the relation between the Madelung setting and the Marsden--Weinstein symplectic structures on knots and membranes.
\end{abstract}

\maketitle

\arxivonly{
\tableofcontents
}

\journalonly{
\subsection*{Statements and Declarations}
The authors have no relevant financial or non-financial interests to disclose.
Data sharing is not applicable to this article as no datasets were generated or analyzed during the current study.

}

\section{Introduction}

In 1927, right after the birth of quantum mechanics, Madelung~\cite{Ma1927} gave a hydrodynamical reformulation of the non-relativistic Schrödinger equation.
It has since been debated if the two formulations are equivalent.
The best known critique against Madelung's hydrodynamical setting was formulated by Wallstrom~\cite{Wa1994,Wa1994b}, who raised two main objections.
First, that the hydrodynamical formulation needs to be augmented with the quantization condition of Takabayasi~\cite{Ta1952}, similar to how Bohr's original atomic model needed an \emph{a priori} quantization condition on the orbits.
Second, that the initial value problem for the hydrodynamical formulation is ill-posed.
Many authors, both mathematicians and physicists, have since argued for and against Wallstroms objections; for details and an overview of the literature we refer to the topical review by Reddinger and Poirier~\cite{RePo2023}.
However, the precise relation between the two formulations has remained clouded. 

In this paper we resolve the Wallstrom controversy.
Namely, we pinpoint the conditions for when the original Schrödinger equation and Madelung's hydrodynamical equations are equivalent with respect to generic, strong solutions. 
It turns out to be a question of geometry rather than analysis, where 
the Schrödinger equation on a general Riemannian manifold $M$ can be approached via the framework of geometric mechanics (cf.~\citet{Ar1989}).
Our solution then rests on two prior results: 
Fusca's~\cite{Fu2017} realization of the Madelung transform as a momentum map relative to a group action on the space of wave functions, and the idea of Chern and Ishida~\cite{ChIs2025} to represent codimension-2 submanifolds as zero-sets of wave functions relative to a prequantum structure. 
With these two ingredients, the equivalence problem becomes a question of the geometry and topology of 
an appropriate group action on wave functions.











\subsection{The Madelung equations}

Let us recall Madelung's hydrodynamical interpretation~\cite{Ma1927} of the Schrödinger equations.
On a Riemannian manifold~$M$, consider the following (possibly non-linear) Schrödinger equation 
\begin{equation}\label{eq:schrodinger}
    \mathrm i \dot\psi = -\Delta\psi + V \psi + f(\lvert \psi \rvert^2)\psi .
\end{equation}
for a wave function $\psi \in  C^\infty(M,\mathbb C)$, where $V\in C^\infty(M,\mathbb R)$ is a potential and $f\in C^\infty(\mathbb R_{\ge 0}, \mathbb R)$ defines a non-linearity.
The Hamiltonian relative to the symplectic structure $2\, \mathrm{Im}\langle \cdot_1,\cdot_2\rangle_{L^2}$ is
\begin{equation*}
    H(\psi) = \int_M \Big( \lvert\nabla \psi \rvert^2 + V\lvert \psi\rvert^2 + F(\lvert \psi\rvert^2) \Big)\,\mu\,,
\end{equation*}
where $F'(r) = f(r)$ and the integration is with respect to the Riemannian volume form $\mu$. 
Madelung realized that the Schrödinger equation~\eqref{eq:schrodinger} can be written as a barotropic hydrodynamical system in the velocity field $v = \mathrm{Im}\big( \frac{\nabla\psi}{\psi}\big) $ and the density function $\rho = \lvert \psi\rvert^2$.
Indeed, in these variables, the Schrödinger equation assumes the form of the Madelung equations
\begin{equation}\label{eq:madelung_eq}
    \dot v + \nabla_v v = - 2\nabla \left( V + f(\rho) - \frac{\Delta\sqrt{\rho}}{\sqrt{\rho}} \right), \qquad \dot\rho + \operatorname{div}(\rho v) = 0.
\end{equation}
 
 The Madelung transform is $\psi=\sqrt{\rho e^{\rm i\theta}}$ with $v=\nabla \theta/2$. Note that the vector field $v$ contains the same information as $\theta$ up to an additive constant. In other words, the Madelung transform takes a wave function $\psi$ to the pair $(\theta, \rho)$, where $\theta$ and $\rho$ are the argument and the modulus of $\psi^2$, with $\theta$  understood up to an additive constant.

There is an obvious problem with the change of coordinates $\psi \leftrightarrow (v,\rho)$, as it leads to singularities whenever the wave function is zero somewhere.
By going from the Lagrangian variables $(v,\rho)$ to the Hamiltonian variables $(\nu,\rho)= (\rho v, \rho)$ the problem is partially resolved.
However, the concern for what happens when $\rho$ attains zero is the core issue underlying Wallstrom's critique, cf.~\cite{Wa1994}.
One can see it already at the level of the Hamiltonians: whereas $H(\psi)$ is perfectly smooth regardless of the zeros of $\psi$, its expression in the variables $(\nu,\rho)$ has singularities at zeros of $\rho$:
\begin{equation}\label{eq:Hamiltonian}
    H(\nu,\rho) = \int_M \left( \frac{\lvert \nu \rvert^2}{\rho} + \lvert \nabla \sqrt{\rho}\rvert^2 + V\rho + F(\rho) \right)\mu\,.
\end{equation}

\medskip

\begin{remark}
The two aforementioned objections by Wallstrom are related to the zero set of the wave function, $\gamma=\psi^{-1}(0)$. 
If $\operatorname{codim}\gamma \ge2$, then for a contour $\ell$ around $\gamma$ the phase $\theta$
has to satisfy the condition $\oint_\ell {\rm d}\theta =4\pi k$ for $k\in \mathbb Z$, since it corresponds to a univalued wave function $\psi=\sqrt{\rho e^{\rm i\theta}}$. 
Thus, solutions to the Schrödinger equation~\eqref{eq:schrodinger} correspond to only a certain subset of  solutions of the Madelung equations~\eqref{eq:madelung_eq}, and the ``quantization condition'', while absent on the Schrödinger side, according to Wallstrom's first objection, seems \emph{ad hoc} on the hydrodynamical side. 

The second objection, also called the ``Wallstrom phenomenon'',  is that if $\gamma$ separates $M$ into connected components, the wave functions can be chosen with a large degree of freedom in different connected components, since while crossing the ``wall of zeros'' $\gamma=\psi^{-1}(0)$, the corresponding wave functions ``forget'' their phases, cf.~\cite{Wa1994b}.
Wallstrom's argument was that a wave function solving the Schrödinger equation can be multiplied by a constant phase in one of the components, which the Madelung equations will not notice, as they rely only on~$\nabla\theta$.

However, if the zeros of a smooth wave function $\psi$ are noncritical, the 
set $\gamma=\psi^{-1}(0)$ is a codimension-2 submanifold of $M$, not of codimension 1 as required for the Wallstrom phenomenon,
so $\gamma$ does not split the manifold in different connected components.
In particular, for a connected $M$ the complement $M\backslash \gamma$ is also connected. 
Since zeros of $\psi$ are generically noncritical, 
the Wallstrom phenomenon of non-uniqueness of solutions is thus generically ruled out. 
In what follows we focus our discussion on the first objection, the  quantization condition
(cf.~Example~\ref{ex:charts} below and the discussion in Remark~\ref{rem:switching}).
\end{remark}

\subsection{Properties of the Madelung transform}
For nonvanishing wave functions, the Madelung transform has remarkable geometric properties.
Indeed, let $\psi$ be  a  nonvanishing (normalized) wave function on a manifold $M$, namely $$\psi\in \WFnonzero \coloneqq C^\infty(M, \mathbb C^*)\cap \{\psi\mid \int_M \lvert \psi\rvert^2\mu = 1 \}.$$ 
Consider the projectivization $\mathbb{P}\WFnonzero$
of the space $\WFnonzero$, which is naturally equipped with the (infinite-dimensional) Fubini--Study symplectic structure and metric.
Also consider the space of (smooth) probability  
densities $\Dens(M)=\{\varrho \in \Omega^n(M) \mid \varrho>0, \, \int_M\varrho=1  \}$, 
as well as its cotangent space $T^*\Dens(M)=\{(\varrho, [\theta]) \mid \varrho\in \Dens(M), \, [\theta]\in C^\infty(M, \mathbb R)/\mathbb R\}$, where the cosets $[\theta]$ stand for phase functions $\theta$ considered modulo an oveall phase shift.

\begin{theorem}[\cite{KhMiMo2019}]\label{KMM}
  For a simply-connected manifold $M$, the Madelung transform, taking a nonvanishing wave function $\psi=\sqrt{\rho e^{\rm i\theta}}$ to $\mathbf{S}(\psi)=(\varrho, [\theta])$ for $\varrho=\rho\mu$
  and regarded as a map $\mathbf{S}\colon \mathbb{P}\WFnonzero \to T^*\Dens(M)$, is
    
    $i)$ a symplectomorphism of the Fubini--Study symplectic structure on $\mathbb{P}\WFnonzero$ and the standard symplectic structure
    on $T^*\Dens(M)$, 
    
    $ii)$ an isometry of the Fubini--Study metric on $\mathbb{P}\WFnonzero$ and the Sasaki--Fisher--Rao metric
    on $T^*\Dens(M)$. 
\end{theorem}

The main results of the present paper provide a general framework  for wave functions with zeros, appropriate group actions and Poisson properties of the Madelung transform on  
arbitrary manifolds, 
while the above theorem becomes a particular case of that setting.
Namely, on the space of wave functions there is a natural group action discovered by Fusca \cite{Fu2017}.
It turns out it is convenient to regard wave functions as half-densities.
Indeed, we will often use the identification $\psi \simeq \hd = \psi\,\sqrt\mu$
by using a reference (nowhere vanishing) volume form $\mu$ on $M$.

 Fusca proved that the semi-direct product group $\DC \coloneqq \mathrm{Diff}(M)\ltimes C^\infty(M, S^1)$ 
 (representing  changes of coordinates on $M$ suplemented by the pointwise phase shifts)
 acts on the space of wave half-densities, while the corresponding momentum map $\Mom\colon C^\infty(M, \mathbb C) \to \mathfrak{dc}^* $  can also be regarded as the Madelung transform, see \cite{Fu2017}. 
 Below we discuss this momentum map in detail.  Throughout the paper, without loss of generality,  we will always be considering {\it normalized wave functions}, i.e., wave functions belonging to 
 the unit sphere
 $S^\infty_1:=\{\psi\in C^\infty(M, \mathbb C)~|~\int_M|\psi|^2\mu=1\}$.
 
  \medskip
  
  In Section \ref{sect:no-zeros} we derive various consequences of the momentum map
 for wave functions without zeros. 
 We show that for an arbitrary manifold $M$ the connected components of 
$\mathbb{P}\WFnonzero$ are enumerated by elements of $H^1(M, \mathbb Z)$. Furthermore,
in Theorem~\ref{thm:no-zeros} it is proved that to {\it each of those connected components 
the momentum map $\Mom$  associates  the 
corresponding coadjoint orbit $\mathcal O_{(\nu, \rho)}\subset \mathfrak{dc}^* $, as well as
it provides the structure of a prequantum $S^1$-bundle
over the orbit.} In particular, this is a Poisson map and it implies the symplectomorphism of the Fubini--Study on $\mathbb{P}\WFnonzero$ and $T^*\Dens(M)$ as described in \cite{KhMiMo2019}.

Section \ref{sect:yes-zeros} deals with wave functions with zeros. The assumption of noncritical zeros 
$\gamma=\psi^{-1}(0)$
implies that $\operatorname{codim}\gamma=2$ and the corresponding densities $\varrho=|\psi|^2$ are Morse-Bott, i.e., they have the simplest type of degeneration on $\gamma$. Due to the Moser type theorem for Morse-Bott densities established in \cite{KhesinVolk25}, one can prove for such densities most of the properties valid for nondegenerate ones. Our second main result in
abbreviated form is the following.

\begin{theorem}{\bf ($=$ Theorems \ref{thm:sing-alpha}$'+$\,\ref{thm:zeros}$'$)}
 The momentum map $\Mom$ for normalized wave functions with noncritical zeros on $\gamma\subset M$
provides the correspondence of their connected components, which are enumerated by $H^1(M, \mathbb Z)$, with  coadjoint orbits $\mathcal O_{(\nu, \rho)}\subset \mathfrak{dc}^* $ and defines $S^1$-bundle
over each orbit in the image of $\Mom$. 

The orbits $\mathcal O_{(\nu, \rho)}$ in the image are coadjoint orbits for elements $\nu=\alpha\otimes \varrho$  where 
densities $\varrho$ are Morse-Bott with degenerations on $\gamma$, while 
the 1-form $\alpha$ has a pole-like singularity on $\gamma$ and
must satisfy the quantization condition:  the periods of $\alpha$ around components of 
$\gamma$ are equal to $\pm 2\pi$ depending on the components' orientations. 
\end{theorem}

Thus the Wallstrom quantization condition on the hydrodynamical side exactly corresponds to the orbits belonging to the image of the momentum map in the dual~$\mathfrak{dc}^* $.

There is yet another phenomenon observed for the quantum mechanics / fluids correspondence:   
a trajectory $\psi_t$ in the symplectic space $C^\infty(M,\mathbb{C})$ is projected by a momentum map $\Mom$ 
to a trajectory in the Poisson space $\mathfrak{dc}^*$ {\it switching between symplectic leaves}, see Remark~\ref{rem:switching} and Figure~\ref{fig:projection}. 
This cannot happen in smooth Hamiltonian systems, so we observe here that the projected image trajectory
approaches the boundary of the symplectic leaf with infinite speed at the critical time. (This manifests itself by $\varrho$ in the denominator of the expression for the Hamiltonian $H$ in the space $\mathfrak{dc}^*$, which implies the infinite variational derivative at zeros of $\varrho$, cf. Equation~\eqref{eq:Hamiltonian}.)
These ``jumps between coadjoint orbits" manifest another curious feature of the hydrodynamical form of the Schr\"odinger equation  not discussed before.

\begin{example}\label{ex:charts}
In a sense, the two features, quantization condition and switching the symplectic leaves, are related to a special choice of coordinate charts, provided by the Madelung transform. Here is a basic analogy illustrating it. 
Consider the harmonic oscillator $\ddot x + x = 0$ with initial data $x_0 > 0$ and $\dot x_0 \in \mathbb{R}$.
The phase space consists of points $(x,\dot x) \in \mathbb{R}^2$. 
We now introduce new coordinates $(y,\dot y)$ via $x = \mathrm{e}^y$ and the corresponding velocity variable $\dot x = \mathrm{e}^y \dot y$.
Expressed in these coordinates, the system becomes $\ddot y + \dot y^2 + 1 = 0$.
As long as the solution fulfills $x(t) > 0$, the original system in $(x,\dot x)$ and the new system in coordinates $(y,\dot y)$ are equivalent.
However, the coordinate chart $(y,\dot y)$ is only local and breaks down when $x(t) \leq 0$.
For the system expressed in $(y,\dot y)$-coordinates, it means that $y(t) \to -\infty$ in finite time.
This is a property of the coordinate chart, rather than the equation. 
An easy fix to the problem is just to change to a different coordinate chart. 
If we take $x = -\mathrm{e}^{y}$ we obtain the same equation for $(y,\dot y)$, but we have to make a ``jump'' in phase space at the transition when $x(t)$ passes from positive to negative values.
In this analogy, the Schr\"odinger equation corresponds to an evolution in $x$-variables, while the Madelung equations correspond to that in the $y$-coordinate chart. While Schr\"odinger solutions remain smooth,
the variables used in the Madelung equations break down (resulting in hydrodynamical blow-up) at instances when the topology of the zero-set of the wave function changes, i.e., zero becomes a critical value of the wave function. 
\end{example}


\subsection{Infinite-dimensional torus actions}
Given the situation described in the analogy, one may ask why to consider the hydrodynamical formulation at all, since the wave function formulation already admits global coordinates.
However, the effective, physical phase space is not linear: due to 
the global phase invariance, the actual phase space is a complex projective manifold $\mathbb{ CP}^\infty  $ and the Schrödinger equation is a Hamiltonian system with respect to its natural symplectic structure ({cf.}~geometric quantum mechanics~\cite{Ki1979,BrHu2001}).
This complex projective space does not admit a global coordinate chart (e.g., in the simplest case of a qubit,  i.e. the quantum version of the classic binary bit, we have $\mathbb{CP}^1 \simeq \mathbb{S}^2$).
In this respect, the Madelung transform provides local canonical coordinates, and thus puts the Schrödinger equation in canonical Hamiltonian form.

Furthermore, we prove that there is a fundamental {\it analogy between finite-dimen\-sional toric geometry and the Madelung momentum action}: the latter can be naturally  viewed as an infinite-dimensional version of the classical convexity theorems, thus giving a partial answer to Atiyah's question about that posed in \cite{atiyah1985}.
We explore this in detail in Section~\ref{sect:finiteMadelung} (see in particular Section~\ref{sect:table}) and give here only a sample of the correspondence:

$(a)$ The unitary group $U(n+1)$ acts on transitively on 
$\mathbb{CP}^n$,  its smallest nontrivial coadjoint orbit. Similarly the semidirect product group $\DC
={\rm Diff}(M)\ltimes C^\infty(M,S^1)$ acts on the projective space of wave functions $\mathbb{P}\WF$.

$(b)$ The maximal torus $\mathbb T^n$ corresponds to the infinite-dimensional phase torus
$\mathbb T^\infty=C^\infty(M,S^1)$
acting by pointwise phase rotations. 

$(c)$ While the image of the momentum map
$\Mom_{\mathbb{T}^n}\colon\mathbb{CP}^n\to(\mathfrak t^n)^{*}\simeq \mathbb R^n$ is the simplex $\Delta^n =
\{x_i\ge0,\ \sum x_i=1\}$, its infinite-dimensional analogue, the image of the momentum map
$\Mom_{\mathbb{T}^\infty}\colon \mathbb{P}\WF\to
{(\mathfrak{t}^\infty})^{*}
\simeq \Omega^n(M)$ 
has the density (or probability)  space 
$
\overline\Dens(M)
=
\{\varrho\ge0,\ \varrho(M)=1\}
$
(after an appropriate completion), with faces of this ``continuous polytope" corresponding to density functions vanishing on various subsets of $M$ and implied adjacency relations.

We refer to Section~\ref{sect:table}
for more details and analogues.

\begin{remark}
It is interesting to compare this correspondence with another one, via a symplectomorphism group $ {\rm Symp}_\omega(M)$
of a toric manifold $M^{2n}$, where the role of the maximal torus is played by equivariant symplectomorphisms ${\rm Symp}_\omega(M, T)$, cf.~\cite{bloch1993schur, mousavi2026}. The latter is  similar to the action of $U(n+1)$ on generic (i.e., ``largest") orbits in $\mathfrak u(n+1)^*$. On the other hand,  the action of the group $\DC$ on wave functions is similar to the action of $U(n+1)$ on the smallest nontrivial coadjoint orbits (which are symplectomorphic to $\mathbb{CP}^n$), namely, orbits  of matrices of rank 1, see Section~\ref{sect:symplectomorphisms}.
\end{remark}


Finally, in Appendix~\ref{sect:MW} we recall the Marsden--Weinstein symplectic structures on the spaces of knots and membranes.  
 We discuss their relation to the symplectic forms of the coadjoint orbits in $\mathfrak{dc}^*$ and the framework  of the Madelung transform, 
as well as to its prequantization discussed in different terms in \cite{ChIs2025}.

\medskip


\newcommand{\acknowledgement}{{\bf Acknowledgements.} The authors would like to thank Anton Izosimov and Ood Shabtai for useful discussions.
The work of BK was partially supported by an NSERC Discovery Grant.
This work of KM was supported by the Swedish Research Council (grant number 2022-03453), the Knut and Alice Wallenberg Foundation (grant numbers WAF2019.0201 and KAW2020.0287), and the Göran Gustafsson Foundation for Research in Natural Sciences and Medicine.
}

\acknowledgement


\section{Wave functions without zeros}\label{sect:no-zeros}

Let $M$ be a connected $n$-dimensional oriented manifold. Our two main objects of study are the space of smooth complex-valued wave functions on $M$ and a semi-direct product group of diffeomorphisms and circle-valued functions acting on them. 
 If the manifold $M$ is not compact one has to confine to the Schwartz space of fast decaying wave functions and, respectively, to diffeomorphisms that are close to the identity at infinity, to make all the integrals discussed below convergent.

 Geometrically, it is natural to think of wave functions $\psi$ as complex valued half-densities $\hd$.
 Throughout the text, we use both conventions, with the relation $\hd = \psi \sqrt{\mu}$ for a fixed background volume form $\mu$.
 Notice, however, that all the developments are independent of the choice of $\mu$ (cf.~\cite[Prop.~4.4]{KhMiMo2019}).

\subsection{The group action on wave functions}
Consider the {\it semi-direct product group} $\DC = \mathrm{Diff}(M)\ltimes C^\infty(M, S^1)$,
where $\mathrm{Diff}(M)$ here and below in this paper stands for the connected component containing the identity of the group of diffeomorphisms of $M$, and $S^1 \simeq \mathbb{R}/2\pi\mathbb{Z}\simeq \mathrm{U}(1)$ stands for the circle or, equivalently, the complex numbers of modulus $1$.\footnote{Notice that the extension part in the semi-direct product group is an infinite-dimensional torus $\mathbb{T}^\infty := C^\infty(M,S^1)$, which is important for the finite-dimensional analogue developed below.}
The {\it Lie algebra} of the group $\DC$ is $\mathfrak{dc} = \mathfrak{X}(M)\ltimes C^\infty(M,\mathbb{R})$. 

\begin{definition}
The group $\DC$ acts on the space of smooth wave functions $C^\infty(M,\mathbb{C})$ via 
\begin{equation*}
    (\varphi,a)\cdot \psi  = \left(\psi\circ \varphi^{-1} \right) \mathrm{Jac}_\mu(\varphi^{-1})^{1/2} \operatorname{exp}(-\mathrm i a)
\end{equation*}    
which is a natural action on \emph{half-densities}:
\begin{equation}
    (\varphi,a)\cdot \hd = \exp(-\mathrm i a) \varphi_*\hd .
\end{equation}
\end{definition}

This action is Hamiltonian with respect to the symplectic structure on $C^\infty(M,\mathbb{C})$,
\begin{equation*}
    \mathrm i\,\mathrm d\psi\wedge \mathrm d\bar\psi = 2 \,\mathrm{Im}\langle \cdot_1,\cdot_2\rangle_{L^2},
\end{equation*}
where $\langle \cdot_1,\cdot_2\rangle_{L^2}$ is the standard Hermitian inner product on wave functions.
Fusca~\cite{Fu2017} proved that the corresponding momentum map $\Mom\colon C^\infty(M, \mathbb C) \to \mathfrak{dc}^* $ is the (inverse) Madelung transform.
Below we define this momentum map and study its properties in detail.

\medskip

Before we describe the map $\Mom$, we recall the structure of the dual $\mathfrak{dc}^* $ 
and introduce coordinates on it.
First note that  the smooth dual $\mathfrak{X}^*(M)$ of the Lie algebra $\mathfrak{X}(M)$
can be identified with the space $\Omega^1(M)\otimes \Omega^n(M)$, where $M$ is of dimension $n$ and the tensor product $\otimes$ is over $C^\infty(M)$. Respectively, the smooth dual $\mathfrak{dc}^*$ of the semidirect-product Lie algebra 
$\mathfrak{dc} = \mathfrak{X}(M)\ltimes C^\infty(M,\mathbb{R})$ is 
$\left(\Omega^1(M)\otimes \Omega^n(M)\right)\oplus \Omega^n(M)$. The pairing is defined as follows: for 
$(v, f)\in \mathfrak{X}(M)\ltimes C^\infty(M,\mathbb{R})$ and 
$(\nu, \varrho)\in\Omega^1(M)\otimes \Omega^n(M)\oplus \Omega^n(M)$ one sets
$$
\langle (\nu, \varrho), (v, f)\rangle :=
\int_M \iota_v \nu  +\int_M f\,\varrho\,.
$$
Below we will also use another coordinate system on an open subset 
of $\mathfrak{dc}^*$ with nonvanishing densities $\varrho$. 
Namely, whenever $\varrho\not=0 $ on $M$ it will be convenient 
to write the elements of $\mathfrak{dc}^*$ in the form 
$(\alpha\otimes\varrho, \varrho)$. 
The coadjoint orbit of $(\nu,\varrho)$ is denoted
$\mathcal{O}_{(\nu, \varrho)}\subset \mathfrak{dc}^*$.

\begin{remark}
    A natural functional-analytic setting for the results presented in this paper is that of tame Fr\'echet spaces, 
    cf.~Hamilton~\cite{Ha1982}. An alternative setting for groups of diffeomorphisms deals with Sobolev $H^s$ completions 
    (or any reasonably strong Banach topology) 
    of the corresponding function spaces~\cite{EbMa1970}. 
    If $s>\dim{M}/2 +1$ then the Sobolev completions of the diffeomorphism groups 
    $\mathrm{Diff}^s(M)$ and $\mathrm{Diff}_\mu^s(M)$ are smooth Hilbert manifolds but not Banach Lie groups since, 
    for example, the left multiplication and the inversion maps are not uniformly continuous in the $H^s$ topology. 

    On the other hand, both $\mathrm{Diff}(M)$ and $\mathrm{Diff}_\mu(M)$ can be equipped with the structure of 
    tame Fr\'echet Lie groups. 
    In this setting, $\mathrm{Diff}_\mu(M)$ becomes a closed tame Lie subgroup 
    of $\mathrm{Diff}(M)$ which can be viewed as a tame principal bundle over the quotient space 
    $\Dens(M) = \mathrm{Diff}(M)/\mathrm{Diff}_\mu(M)$ 
    of either left or right cosets. 
    Furthermore, the tangent bundle $T\mathrm{Diff}(M)$ over $\mathrm{Diff}(M)$ is also a tame manifold. 
    However, since the dual of a Fr\'echet space, which itself is not a Banach space, is never a Fr\'echet space, 
    to avoid working with currents on $M$ it is expedient to restrict to a suitable subset of the (full) cotangent bundle 
    over $\mathrm{Diff}(M)$. 

    More precisely, consider the tensor product $T^\ast M \otimes \Lambda^n M$ of the cotangent bundle and 
    the vector bundle of $n$-forms on $M$ and define another bundle over $\mathrm{Diff}(M)$ 
    whose fibre over $\varphi \in \mathrm{Diff}(M)$ is the space of smooth sections of the pullback bundle 
    $\varphi^{-1}(T^\ast M \otimes \Lambda^n M)$ over $M$. 
    We will refer to this object as (the smooth part of) the cotangent bundle of $\mathrm{Diff}(M)$ 
    and denote it also by $T^\ast \mathrm{Diff}(M)$. 
    We will write $\mathfrak{X}^*(M) = T^*_\mathrm{id}\mathrm{Diff}(M)$ and $\mathfrak{X}^{**}(M) = \mathfrak{X}(M)$.
    Throughout the paper we will assume that derivatives of various Hamiltonian functions 
    can be viewed as maps to the smooth cotangent bundle of the phase space.
\end{remark}

Now we are ready to describe the momentum map $\Mom$. 

\begin{theorem}[\cite{Fu2017}]
    The momentum map $\Mom$ for the action of the group $\DC$ on the space of wave functions 
    $ C^\infty(M,\mathbb{C})$ has the following explicit form:
\begin{equation*}
    \Mom\colon C^\infty(M,\mathbb{C}) \to \mathfrak{dc}^* , \qquad
    \hd \mapsto \Big( \operatorname{Im}(\bar\hd \, \mathrm d\hd), \bar\hd \hd \Big) .
\end{equation*}
\end{theorem}

Note that the operation $\mathrm{Im}(\overline{\hd}\,\mathrm d\hd)$ for $\hd = \psi\sqrt{\mu}$ is well-defined on half-densities, see \cite{KhMiMo2019}.

\begin{proof}
    The action of $(v,f)\in\mathfrak{dc}$ on the half-density $\hd=\psi\sqrt{\mu}$ is $-L_v(\psi\sqrt{\mu}) - \mathrm i f \psi\sqrt{\mu}$.
    Take the Hamiltonian 
    \begin{equation*}
        H_{(v,f)}(\psi) = \int_M \mathrm{Im}(\bar\psi \iota_v \mathrm d\psi)\mu + \int_M f \bar\psi\psi \, \mu.
    \end{equation*}
    We then have, with $\psi' \coloneqq \frac{d}{d\epsilon}\psi_\epsilon$,
    \begin{align*}
        \frac{d}{d\epsilon}H_{(v,f)}(\psi_\epsilon) &= \mathrm{Im}\int_M (\bar\psi' \mu L_v \psi + \bar\psi \mu L_v \psi') + 2\, \mathrm{Re}\int_M f \bar\psi \psi' = 
        \\ &= 2\, \mathrm{Im}\int_M \bar\psi' (\mu L_v \psi + \frac{1}{2}\psi L_v \mu) + 2\, \mathrm{Im}\int_M \mathrm i f \bar\psi\psi' \mu = 
        \\ &= 2\, \mathrm{Im}\int_M ( -L_v \bar\psi\sqrt{\mu} + \mathrm i f \bar\psi)\psi'\sqrt{\mu} 
        = 2\,\mathrm{Im}\langle (-L_v\psi-\mathrm i f \psi)\sqrt{\mu},\psi'\sqrt{\mu}\rangle_{L^2},
    \end{align*}
    which proves that the Hamiltonian vector field is the infinitesimal action.
\end{proof}


This momentum map is equivariant, namely 
\begin{equation*}
    \Mom\big((\varphi,a)\cdot\psi\big) = \mathrm{Ad}^*_{(\varphi,a)}\big(\Mom(\psi)\big)
\end{equation*}
where
\begin{equation*}
    \mathrm{Ad}^*_{(\varphi,a)}(\nu, \varrho) = \big(\varphi_*\nu - \mathrm d a \otimes \varphi_*\varrho, \; \varphi_*\varrho\big).
\end{equation*}
The latter means, in particular, that $\DC$-orbits in $C^\infty(M,\mathbb{C})$ are mapped to coadjoint orbits in $\mathfrak{dc}^*$.
However, the momentum map $\Mom$ is not a submersion, not all coadjoint orbits are obtained in the image.
\medskip

Our first main goal will be {\it to describe which coadjoint orbits in $\mathfrak{dc}^*$ are contained in the image of the momentum map~$\Mom$}. 
This seemingly modest task will allow us not only interpret the Madelung transform as a geometric prequantization of those coadjoint orbits, but also to establish its Kähler properties, previously known from a different perspective, and to give answers to Wallstrom's objections to the hydrodynamical form of the quantum mechanics.
As we mentioned in the introduction,  we confine to {\it normalized wave functions}, i.e., wave functions belonging to the unit $L^2$-sphere
 $S^\infty_1=\{\psi\in C^\infty(M, \mathbb C)~|~\int_M|\hd|^2=1\}$.

In this section we consider  the case of normalized {\it nonvanishing wave functions},  
$$
\psi \in \WFnonzero \coloneqq C^\infty(M,\mathbb{C}^*)\cap S^\infty_1\,,
$$
where $\mathbb{C}^*=\mathbb{C}\backslash \{ 0\}$.

\subsection{Topology of the action}
Before we describe symplectic features  of the momentum map, we give an account of  topological properties of the $\DC$-group action on wave functions.


\begin{theorem}\label{thm:non-zero-trans}
The connected components $\WFnonzero^{cc}$ of the space $\WFnonzero$ of normalized nonvanishing wave functions are
enumerated by the elements of $H^1(M,2\pi\mathbb{Z})$. 
The action of the group $\DC = \mathrm{Diff}(M)\ltimes C^\infty(M, S^1)$
on half-densities is transitive on each connected component of $\WFnonzero$.
\end{theorem}

The proof is a combination of several lemmas. 

\begin{lemma}\label{lem:closed-alpha}
    Given $\psi \in \WFnonzero$, define the 1-form $\alpha= \frac{\mathrm i}{2} (\frac{\mathrm d\bar\psi}{\bar\psi}-\frac{\mathrm d\psi}{\psi} )$. Then $\alpha$ 
    is a closed real 1-form on $M$. Furthermore, for $\psi$ written ``in polar coordinates" as
    $\psi = \sqrt{\rho \mathrm{e}^{\mathrm i \theta}}$, we locally have that $\alpha = \mathrm d \theta/2$.
\end{lemma}

\begin{proof}
 From definition of $\alpha$ we immediately have $\mathrm{d} \alpha =  0$.
    If $\psi = \sqrt{\rho \mathrm e^{\mathrm i\theta}}$, then $\rho=\bar\psi \psi$ and hence locally
    $\mathrm d \psi = \frac{1}{2\psi} (\mathrm{e}^{\mathrm i\theta} \mathrm d \rho + \mathrm{i}\rho \mathrm{e}^{\mathrm i \theta}\mathrm d \theta) = \frac{\psi}{2}(\mathrm d\rho/\rho + \mathrm i \,\mathrm d \theta)$.
    Similarly, $\mathrm d\bar\psi = \frac{\bar\psi}{2}(\mathrm d\rho/\rho-\mathrm i \,\mathrm d\theta)$.
    Then it follows that $\alpha = \mathrm d \theta/2$.    
\end{proof}

\begin{lemma}\label{lem:integral-homology}
     Given $\psi \in \WFnonzero$, the cohomology class $[\alpha]$ of the closed 1-form $\alpha$ belongs to $H^1(M,2\pi\mathbb{Z})$: this integer cohomology class counts the rotation number for the phase of $\psi$ over any nontrivial loop on $M$.
\end{lemma}

\begin{proof}
    Fix an arbitrary $x_0 \in M$.
    We can now construct a multivalued function $\theta$ on $M$ by setting $\theta(x) = 2\int_{\ell}\alpha$ where $\ell$ is a curve such that $\ell(0) = x_0$ and $\ell(1) = x$. 
    Since $\alpha$ is a closed 1-form, but not necessarily exact, the integral is the same on  homotopic paths, while the existence of noncontractible curves may lead to $\theta$ being multivalued. If $M$ has trivial first cohomology, the function $\theta$ is univalued.
For an arbitrary $M$ the multivaludeness of $\theta$ is defined by periods of $\alpha$:
one has $\int_\ell \alpha - \int_{\tilde\ell}\alpha = \int_{\ell-\tilde\ell} \alpha$. 

Now define a complex-valued function $\psi$ on $M$ by setting 
$\psi := \sqrt{\rho \mathrm e^{\mathrm i \theta}}=\sqrt{\rho }\mathrm e^{\mathrm i \theta/2}$. 
The function $\psi$ is univalued on $M$ if and only if the periods of $\alpha$ belong to $2\pi \mathbb Z$, i.e., $[\alpha] \in H^1(M,2\pi\mathbb{Z})$.
\end{proof}

\begin{remark}\label{rem:topological_proof}
Note that Lemma \ref{lem:integral-homology} is actually a purely topological statement equivalent to the following result:

\begin{proposition}
    $\pi_0(\WFnonzero)\cong H^1(M,2\pi\mathbb{Z})$.
\end{proposition}

\begin{proof}
The space $\WFnonzero $, as well as the space $ C^\infty(M,\mathbb{C}^*)$ of non-normalized smooth sections $\psi$ in 
the trivial $\mathbb{C}^*$-bundle over $M$, is homotopically equivalent to the space $C^\infty(M,S^1)$ of smooth sections $f:=\psi/|\psi|$ in $S^1$-bundle. The latter are classified by the homomorphisms $\pi_1(M)\to \pi_1(S^1)\cong \mathbb{Z}$.
Indeed, for any smooth map $f:M\rightarrow S^{1}$, its homotopy class corresponds to a unique element $[f]\in H^{1}(M,2\pi\mathbb{Z})$.

The connection is established by considering the induced map on the fundamental groups, $f_{*}:\pi_{1}(M)\rightarrow \pi_{1}(S^{1})\cong \mathbb{Z}$. 
The elements of 
$$
\text{Hom}(\pi _{1}(M),\mathbb{Z})=\text{Hom}(\pi_{1}(M)/[\pi_1, \pi_1],\mathbb{Z})=\text{Hom}(H_{1}(M,\mathbb{Z}),\mathbb{Z})
$$ 
are in bijection with the elements of $H^{1}(M,\mathbb{Z})\cong \text{Hom}(H_{1}(M,\mathbb{Z}),\mathbb{Z})$ by Poincar\'e duality and the universal coefficient theorem.
\end{proof}

As an example, if $M=S^{1}$, the maps are classified by the degree (winding number), which is an integer: $H^{1}(S^{1},\mathbb{Z})\cong \mathbb{Z}$. 
\end{remark}
\medskip

Now return to the proof of Theorem~\ref{thm:non-zero-trans}.
\begin{proof}
It remains to prove the transitivity part of the theorem, i.e., to show that any path $\psi(t)$  in  $\WFnonzero$ can be traced by the $\DC$-group action.

First note that for the nonvanishing  density $\varrho = |\hd|^2$ 
 the only invariant of the $\mathrm{Diff}(M)$-action  is its total volume, $\int_M \varrho=1$, as the consequence of Moser's lemma. Hence there exists $\varphi\in \mathrm{Diff}(M)$ such that $\varrho(1) = \varphi_*\varrho(0)$. As before, we associate the family of 1-forms $\alpha(t)$ to the family of wave functions $\psi(t)$. 
    It remains to find $a\in C^\infty(M,\mathbb{R})$ such that $\alpha(1) = \varphi_*\alpha(0) - \mathrm{d}a$.
    This is possible if and only if $\alpha(0)$ and $\alpha(1)$ have the same cohomology class $[\alpha(0)] = [\alpha(1)]$, which in turn is true since they are homotopic via $t\mapsto \alpha(t)$, while $[\alpha]$ must remain in the same integer class of $H^1(M,2\pi\mathbb{Z})$.
    This proves that the action is transitive on the connected components of $\WFnonzero$.
\end{proof}

\begin{figure}
    \centering
    \includegraphics[scale=0.8]{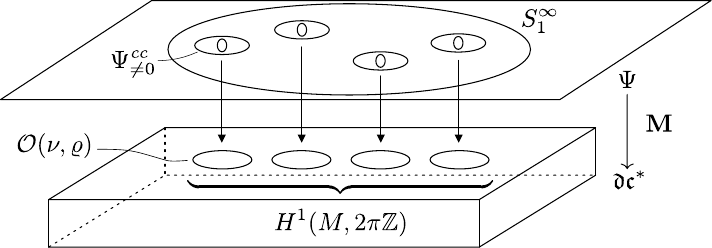}
    \caption{The $\Mom$-projection maps connected components $\WFnonzero^{cc}$ of normalized nonvanishing wave functions in     $ \Psi = C^\infty(M,\mathbb C)$ to coadjoint orbits in $\mathfrak{dc}^*$. The latter  are enumerated by elements of $H^1(M, 2\pi\mathbb Z)$. These prequantum $S^1$-bundles over orbits can be thought of as infinite-dimensional analogues of the Hopf fibration.
 }\label{fig:gen_hopf}
\end{figure}

\subsection{Symplectic geometry of the action}
Now we return to the momentum map $\Mom$
of the semidirect product group $\DC$ given by 
$\Mom\colon \psi \mapsto ( \operatorname{Im}(\bar\hd \, \mathrm d\hd), \bar\hd \hd )$.
The following theorem is strengthening  the results in~\cite{KhMiMo2019} on the Madelung transform of nonvanishing functions, which dealt with the case of a simply connected manifold $M$. 
Now let $M$ be an arbitrary $n$-dimensional oriented manifold. 

\begin{theorem}\label{thm:no-zeros}
Let $(\nu,\varrho) = (\alpha\otimes \varrho,\varrho)$ be an element in $\mathfrak{dc}^*$ with $\varrho >0$ and let $\mathcal{O}_{(\nu, \varrho)}\subset \mathfrak{dc}^*$ be its coadjoint orbit. 
Then the following statements hold:

\begin{enumerate}[label=(\alph*),leftmargin=*]
\item The orbits $\mathcal{O}_{(\nu, \varrho)}$ in the image of $\WFnonzero$ under the momentum map $\Mom$ are in 1-1 correspondence with the 
connected components of $\WFnonzero$ 
enumerated by elements of $H^1(M,2\pi\mathbb{Z})$;

\item 
The orbit $\mathcal{O}_{(\nu, \varrho)}$
belongs to the image 
if and only if 
$\alpha$ is a closed 1-form on $M$ of integer cohomology class,  
 $[\alpha]\in H^1(M,2\pi\mathbb{Z})$;

\item Each connected component $ \WFnonzero^{cc} $ of $ \WFnonzero $ of nonvanishing wave functions $\psi$ in the unit sphere is mapped by $\Phi$ to a coadjoint orbit $\mathcal{O}_{(\nu, \varrho)}$ with the normalization condition $\int_M \varrho=1$. Each of these coadjoint orbits can be understood as an open part of the infinite-dimensional projective space $\mathbb{P}\WFnonzero$ with Fubini--Study symplectic form, and it is naturally symplectomorphic to $T^*\Dens(M)$, where the space of densities is 
$$
\Dens(M)\coloneqq\{\varrho\in \Omega^n(M)\mid\varrho>0, \, \int_M \varrho=1\} \,;
$$

\item Each connected component $ \WFnonzero^{cc} $  is an $S^1$-bundle with the base being the corresponding coadjoint orbit 
$\mathcal{O}_{(\nu, \varrho)}$ and whose bundle projection is the map 
$$
\Mom\colon \WFnonzero^{cc} \to \mathcal{O}_{(\nu, \varrho)}\,.
$$ 
This is a prequantum bundle  for each coadjoint orbit $\mathcal{O}_{(\nu, \varrho)}$
with $[\alpha]\in H^1(M, 2\pi\mathbb{Z})$; 

\item The 1-form $\lambda$ defined by 
\[
    \lambda_\psi[\dot\psi] = \mathrm{Im}\int_M \overline{\psi}\dot\psi
\]
can be regarded as a Liouville 1-form on this prequantum bundle;

\item The bundle projection $\Mom$ is a Riemannian submersion of the $L^2$ metric on $\WFnonzero$ to (the extension of) the 
Sasaki--Fisher--Rao metric on the orbits $\mathcal{O}_{(\nu, \varrho)}$, as  explained below.
\end{enumerate}
\end{theorem}

The multifaceted theorem is illustrated in Figure~\ref{fig:gen_hopf}.
We will prove it as a result of several lemmas of independent interest.

\begin{lemma}\label{lem:preimage}
    Let $\psi \in \WFnonzero$.
    Then $\Mom(\psi) = (\alpha\otimes \varrho,\varrho)$ where $\alpha= \frac{\mathrm i}{2} (\frac{\mathrm d\bar\psi}{\bar\psi}-\frac{\mathrm d\psi}{\psi} )$ is a closed real 1-form on $M$.
    Conversely,   given $(\alpha\otimes\varrho,\varrho) $ in the image of $\Mom$ for some $\psi\in\WFnonzero$,  the full pre-image is  $\Mom^{-1}(\alpha\otimes\varrho,\varrho) = \{ \mathrm e^{\mathrm i c} \psi\mid c\in \mathbb{R}\}$.
\end{lemma}

\begin{proof}
    Since $\mathrm{Im}(\bar\hd\, \mathrm d\hd) = -\frac{\mathrm i}{2}(\bar\hd \mathrm d\hd - \hd \mathrm d\bar\hd) 
    = \frac{\mathrm i}{2} \bar\hd\hd \otimes (\frac{\mathrm d\bar\psi}{\bar\psi}-\frac{\mathrm d\psi}{\psi} )$ and $\bar\hd\hd = \varrho$ we get that 
    $\Mom(\psi) = ( \operatorname{Im}(\bar\hd \, \mathrm d\hd), \bar\hd \hd ) = (\alpha\otimes\varrho,\varrho)$.
    Conversely, since $\varrho = |\hd|^2$, the kernel for the map into the second component of $\Mom$ must be of the form $\mathrm e^{\mathrm i f}\psi$ for some real-valued function $f$.
    The result now follows since for the first component $\Mom_1$ one has $\Mom_1(\mathrm e^{\mathrm i f}\psi) = \operatorname{Im}(\overline{\mathrm e^{\mathrm i f}\hd}\mathrm d (\mathrm e^{\mathrm i f}\hd)) = \operatorname{Im}(\bar\hd \mathrm d \hd) + \varrho\, \mathrm d f = \Mom_1(\psi) + \varrho\, \mathrm d f$, which implies $\mathrm d f = 0$ so $f$ must be a constant function. \end{proof}

\begin{proof}[Proof of Theorem~\ref{thm:no-zeros}]
Lemmas \ref{lem:closed-alpha}, \ref{lem:integral-homology}, and \ref{lem:preimage} prove the statements 
(a)-(d) except for the symplectomorphism part with $T^*\Dens(M)$ in (c).
The latter, together with (f), was proved for a simply-connected $M$ in \cite{KhMiMo2019}, see Theorem~\ref{KMM} above. Recall that for a simply-connected manifold $M$ there is a symplectomorphism $\mathbf{S}\colon \mathbb P\WFnonzero\to T^*{\rm Dens}(M)$. Moreover, in the simply-connected case, we globally have $\alpha = \mathrm d\theta$ and the statement is that $\mathcal{O}_{(\mathrm d\theta\otimes\varrho,\varrho)}$ is symplectomorphic to $T^*\Dens(M)$ via the mapping 
\begin{equation*}
    T^*\Dens(M) \ni (\varrho,[\theta])\mapsto (\mathrm d\theta\otimes\varrho,\varrho)\in \mathcal{O}_{(\mathrm d\theta\otimes\varrho,\varrho)}.
\end{equation*}
Indeed, the latter map can be regarded as the composition  $\Mom\circ \mathbf{S}^{-1}\colon T^*\Dens(M)\to \mathfrak{dc}^*$ 
(where the map $\Mom$ can be naturally understood as defined on  $\mathbb P\WFnonzero$ thanks  to Lemma~\ref{lem:preimage}) and  it is explicitly given by
\begin{equation}
    \langle\Mom\circ \mathbf{S}^{-1}(\varrho,\theta), (v,f) \rangle = \langle \theta, -L_v\varrho \rangle + \langle f, \varrho\rangle = \langle (\mathrm d\theta\otimes \varrho,\varrho), (v,f)). 
\end{equation}
Since the action on $T^*\Dens(M)$ is transitive, it follows that the image of $\Mom$ consists of one coadjoint orbit $\mathcal{O}_{(\mathrm d\theta\otimes\varrho,\varrho)}$ and that it is a symplectomorphism.

Assume now that $M$ is not necessarily simply-connected.
For a closed but not exact 1-form $\alpha_0$ on $M$, consider  the mapping 
\begin{equation*}
    \mathcal{O}_{(\mathrm d\theta\otimes\varrho,\varrho)}\ni (\mathrm d\theta\otimes\varrho,\varrho)\mapsto ((\mathrm d\theta + \alpha_0)\otimes\varrho,\varrho) \in \mathcal{O}_{(\alpha\otimes\varrho,\varrho)} 
\end{equation*}
for a closed 1-form $\alpha:=\mathrm d\theta + \alpha_0$.
Since this momentum shift $\mathrm d\theta\mapsto \mathrm d\theta + \alpha_0$ is a symplectomorphism, $\mathcal{O}_{(\alpha,\varrho)}$ is also symplectomorphic to $T^*\Dens(M)$.
\smallskip

To prove (e) we first recall that 
a {\it prequantum bundle for a symplectic manifold} $(S, \omega)$ is  
understood as an $S^1$-bundle equipped with an $S^1$-invariant 1-form
$\lambda$ satisfying $\rm d\lambda=\omega$ (and nontrivial in the $S^1$-direction). 

Now note that the 1-form $\lambda_\psi$ is linear in $\psi$.
Thus, its differential is the 2-form
\begin{equation*}
    (\dot\psi_1,\dot\psi_2)_{\psi} \mapsto 2\,\mathrm{Im}\int_M \overline{\dot\hd_1}\dot\hd_2,
\end{equation*}
which is the symplectic structure descending to the Fubini--Study one on the projectivization. Its nontriviality and $S^1$-invariance
immediately follows from its restriction to any finite-dimensional subspace: the form $\lambda:={\rm Im}\,\bar z\mathrm d z$ in $\mathbb C^{2k+2}$ defines the standard contact structure in the Hopf bundle $S^{2k+1}\to \mathbb{CP}^k$, while its differential $\rm d\lambda$ descends to the Fubini--Study symplectic structure.

For the Riemannian property in statement (f), we notice that the tangent space $T_{(\nu,\varrho)}\mathcal{O}_{(\nu,\varrho)}$ for $\nu=\alpha\otimes\varrho=(\mathrm d\theta+\alpha_0)\otimes\varrho$
is naturally identified with the tangent space $T_{(\mathrm d\theta\otimes\varrho,\varrho)}\mathcal{O}_{(\mathrm d\theta\otimes\varrho,\varrho)}\simeq T_{(\varrho,[\theta])}T^*\Dens(M)$ by the same shift $\mathrm d\theta\mapsto \mathrm d\theta + \alpha_0$. 
The expression for the Sasaki--Fisher--Rao metric depends only on $\mathrm d\theta$ and $\varrho$, but not on
$\alpha_0$, while the shift 
allows one to define this metric on other orbits. Indeed, regardless of  $\alpha_0$, one has the same expression for the Sasaki--Fisher--Rao metric
\begin{equation*}
    \langle (\dot\varrho,\dot\theta), (\dot\varrho,\dot\theta)\rangle_{(\varrho,\theta)} = \frac{1}{2}\int_M \Big(\frac{\dot\varrho^2}{\varrho} + \dot\theta^2 \varrho\Big),
\end{equation*}
where $\dot\theta \in [\theta]$ is chosen so that $\int_M \dot\theta\varrho = 0$.
The relation to the Fubini--Study metric is then established  as in \cite{KhMiMo2019}, namely if $\psi = \sqrt{\rho \mathrm e^{\mathrm i\theta}}$ then $\dot\psi = \frac12(\dot\varrho/\varrho + \mathrm i \dot\theta)\psi$, so
\begin{equation*}
    2\, \mathrm{Re}\int_M \overline{\dot\hd}\dot\hd = \frac{1}{2}\int_M \Big( \frac{\dot\varrho}{\varrho} - \mathrm{i}\dot\theta \Big)\Big( \frac{\dot\varrho}{\varrho} + \mathrm{i}\dot\theta \Big)\varrho,
\end{equation*}
which gives the result.
\end{proof}



\section{Wave functions with zeros}\label{sect:yes-zeros}
\subsection{Differential topology of half-densities with zeros}

Consider now the case when $\psi$ has zeros on a set $\gamma \subset M$. We assume that 0 is a noncritical value 
of a complex-valued function $\psi$, so $\gamma=\psi^{-1}(0)$ is a smooth submanifold of $M$ of codimension 2. 
If we consider the complement $M\backslash \gamma$, the setting is not so different from the previous section (wave functions without zeros). However, a few additional results are needed first.

\begin{figure}
    \includegraphics{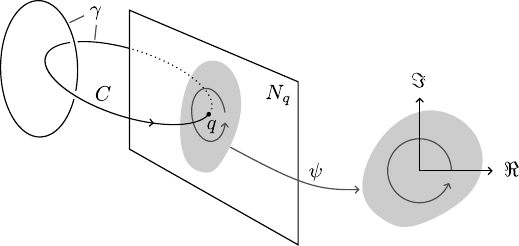}
    \caption{The orientation of any connected component $C\subset\gamma$ is inherited from the orientation of $\mathbb{C}$ since the wave-function $\psi$ restricted to a transversal $N$ is a local diffeomorphism on $\mathbb{C}$ near the zero-set $\gamma$.
    }\label{fig:orient}
\end{figure}

\begin{proposition}
    \phantom{hej}
    \begin{enumerate}[label=(\alph*),leftmargin=*]
        \item Every connected component $C$ of $\gamma$ has a natural orientation.
        \item The density $\varrho=|\hd|^2$ is a Morse-Bott volume form on $M$, i.e., $\varrho$ vanishes on $\gamma$ and has a ``quadratic degeneration in the transversal to $\gamma$ directions''.
        \item The submanifold $\gamma$ is a boundary, $[\gamma]=0\in H_2(M,\mathbb{R})$.
    \end{enumerate}
\end{proposition}

\begin{proof}
    (a) Fix any connected component $C$ of $\gamma$ and consider a transversal $N_q$ to $C$ at a point $q\in C$.
    Then, due to the noncritical property, the restriction $\psi|_{N_q}$ is a local diffeomorphism from a neighborhood of $q\in N_q$ to neighborhood of $0\in \mathbb{C}$.
    Thus, $N_q$ inherits the orientation of $\mathbb{C}$, see Figure~\ref{fig:orient}.
    Furthermore, this orientation of $N_q$
    cannot change along $C$ since  $\psi$ is smooth, so this defines a coorientation of $C$ in $M$.
    Since $M$ is oriented, it also defines an orientation of $C$, as claimed.

     
    (b) The nondegeneracy of $\psi$ on the zero set $\gamma$ implies that $\varrho=|\hd|^2$ has a nondegerate Morse minimum 
   on $\gamma\subset M$ and hence it is a Morse-Bott volume form on $M$, as stated. 

    (c) One can use the hypersurface $\Gamma\subset M\backslash \gamma$ on which $\psi$ has any prescribed generic phase $\theta_0$, e.g., $\theta_0=0\in S^1$. Then the codim 2 submanifold $\gamma=\{\psi=0\}$ is the boundary of a codim 1 submanifold $\Gamma:=\{\mathrm{arg}\,\psi=\theta_0 \mod 2\pi\}$: $\gamma=\partial \Gamma$.
  For a compact $M$ map $\psi: M\to \mathbb{C}$ has a bounded image and $\Gamma$ (or a its small perturbation to make it nonsingular) is a bounded oriented submanifold with boundary.
   \end{proof}

Consider now the set of all normalized wave functions with noncritical zero set $\gamma$, namely
 \begin{equation*}
    \WFgamma := \{ \psi \in C^\infty(M,\mathbb{C})\mid \int_M|\hd|^2=1, \,\psi^{-1}(0) = \gamma \;\,\text{is a noncritical submanifold} \}.
 \end{equation*}
We also consider the union of all wave functions whose noncritical zeros sets are diffeomorphic to $\gamma$ by means of a diffeomorphism of $M$:
 \begin{equation*}
    \WFunion := \{ \psi \in C^\infty(M,\mathbb{C})\mid \int_M|\hd|^2=1, \,\exists \varphi\in\mathrm{Diff}(M) \text{ such that } \psi \in \WFgamma[\varphi(\gamma)]\}.
 \end{equation*}

Below we prove that given $\gamma$ as before, the connected components of the space $\WFunion$ are in correspondence with the elements of $H^1(M,2\pi\mathbb{Z})$.

 \begin{proposition}\label{prop:trans}
    Let $\WFunion[\gamma]^{\it cc}$ denote a connected component of $\WFunion$.
    The semi-direct product group $\DC \coloneqq \mathrm{Diff}(M)\ltimes C^\infty(M, S^1)$ then acts transitively on each $\WFunion[\gamma]^{\it cc }$.
 \end{proposition}

\begin{proof}
The proof is based on the following Moser-type theorem for Morse-Bott volume forms.\footnote{While in
\cite{KhesinVolk25} it was proved for a compact $M$, the theorem allows an immediate extension to a compact $\gamma$ in a noncompact $M$ without the volume constraint.}

\begin{theorem}[\cite{KhesinVolk25}]\label{thm:KV25}
Any two Morse-Bott volume forms on a compact $M$ with the same set of zeros $\gamma\subset M$ of codimension 2
and the same total volume over $M$
are diffeomorphic by means of a diffeomorphism fixed on $\gamma$.
\end{theorem}

To prove transitivity of the $\DC$-action we consider two wave functions $\psi_1$ and $\psi_2$ with diffeomorphic zero sets.
By applying a diffeomorphism of $M$ we can assume that their zero sets coincide with $\gamma$. Furthermore, by applying Theorem \ref{thm:KV25} and composing with another diffeomorphism fixed on $\gamma$ one can assume that upon
application of a diffeomorphism $\varphi$ the two wave functions have the same 
pointwise absolute values, i.e., they can be written as 
$\tilde \psi_1=\sqrt{\rho e^{\mathrm i \theta_1}}$ and 
$\tilde \psi_2=\sqrt{\rho e^{\mathrm i \theta_2}}$. 
Then the ratio of the wave functions can be expressed as 
$\tilde \psi_1/\tilde \psi_2=e^{\mathrm i (\theta_1-\theta_2)/2}$. 
The argument functions $\theta_1$ and $\theta_2$ change by the same quantity when going around any 
connected component of $\gamma$. Then the difference $\theta_1-\theta_2$ is 
a well-defined multivalued function on $M$. By the assumption that $\psi_1$ and $\psi_2$ 
belong to the same connected component $\WFunion[\gamma]^{\it cc}$, 
this difference is actually a univalued smooth function on $M\backslash \gamma$, 
which we denote by $a:=\theta_1-\theta_2$. Furthermore, by the implicit function theorem, 
since $\psi_1$ and $\psi_2$ have nondegenerate linear parts in the normal to $\gamma$, by an appropriate choice of coordinates in the normal, $\psi_1$ can be made linear, while $\psi_2$, having the same degree and  the same pointwise absolute values, 
can differ from from $\psi_1$ by a smooth phase $e^{\mathrm i a}$
where the function $a$ extends smoothly to the points of $\gamma$.
This  implies that $\psi_2=(\varphi,a)\cdot \psi_1$, that is one of the wave functions is obtained from the other by the action of a diffeomorphism $\varphi$ and a smooth phase correction $a$, i.e., by the action of the element $(\varphi,a)\in \DC$.   
\end{proof}

The following special property is yet another way to visualize the relation between a wave function  
$\psi$ on $M$ and its regular zeros $\gamma$. 
Similarly to the previous section define the 1-form $\alpha=\frac {\rm i}{2} \left(\frac{\rm d\bar\psi}{\bar\psi}-\frac{\rm d\psi}{\psi} \right)$  outside the zero-set of $\psi$. 

\begin{theorem}\label{thm:sing-alpha}
Given a wave function $\psi$ with regular zeros on $\gamma\subset M$
consider the singular 1-form $\alpha=\frac {\rm i}{2} \left(\frac{\rm d\bar\psi}{\bar\psi}-\frac{\rm d\psi}{\psi} \right)$.
Then $\rm d \alpha=2\pi\delta_\gamma$, where $\delta_\gamma$ is the de Rham current 
with support on $\gamma\subset M$ ``counting" the number of intersections with it.
\end{theorem}

\begin{proof}
First recall that on $M\backslash \gamma$ the 1-form $\alpha$ is closed (and locally $\alpha=\frac{1}{2}\rm \mathrm d\theta$), hence the 2-form $\rm d\alpha=0$ and ${\rm supp}\, \rm d\alpha\subset \gamma$.
The form $\alpha$ has a pole-like singularity at $\gamma$. 
Now consider a simple contour  $\ell$ around a component of $\gamma$, and let a disk $D$ has boundary $\partial D=\ell$.
Then 
$$
\int_D \mathrm{d}\alpha = \int_{\partial D} \alpha =\frac 12\int_\ell \mathrm{d}\theta = 2\pi\cdot \#(D\cap \gamma)=2\pi\int_D\delta_\gamma\,,
$$
where $ \#(D\cap \gamma)$ is the signed number of intersections of  $D$ and $\gamma$ according to their orientations. (Note that since $\psi=\sqrt{\rho e^{\rm i\theta}}$ one full rotation of the argument of $\psi$
around 0 corresponds to $4\pi$-change in $\theta$.)
\end{proof}


\subsection{Symplectic geometry of the action on wave functions}
We now return again to Fusca's momentum map 
\begin{equation*}
    \Mom\colon \psi \mapsto \Big( \operatorname{Im}(\bar\hd \, \mathrm d\hd), \bar\hd \hd \Big) 
\end{equation*}
for the Hamiltonian action of the semi-direct product group $\DC$ on wave functions~$\psi$ (or equivalently on half-densities $\hd=\psi\sqrt\mu$). 
The latter formula is valid  on the whole of $M$. On the other hand, for $\psi = \sqrt{\rho e^{\mathrm i \theta}}$  the previously used representation $\Mom(\psi) = (\alpha\otimes  \varrho, \varrho)$ with a smooth closed 1-form $\alpha=\frac {\rm i}{2} \left(\frac{\rm d\bar\psi}{\bar\psi}-\frac{\rm d\psi}{\psi} \right)$ is valid outside the zero-set of $\psi$ (or equivalently, outside the zeros of $\varrho$). 
Recall that one has $\alpha=\mathrm d \theta/2$ with a multivalued phase function $\theta$,
locally well-defined (modulo an additive constant) outside the zero-set $\gamma$. 
\medskip

Again, as in the previous section, our goal is to describe the coadjoint orbits $\mathcal{O}_{(\nu, \varrho)}$ contained in the image of the  momentum map  $\Mom\colon \WFunion \to \mathfrak{dc}^*$, i.e.
the $\Mom$-image of all normalized wave functions with noncritical zeros diffeomorphic to $\gamma$. 
We will also need the notation 
$$
{\rm Dens}_\gamma(M):=\{\varrho\in \Omega^n(M)~|~\varrho\ge 0, \,\int_M\varrho=1, \,\varrho(\gamma)=0
\}$$ for all normalized Morse-Bott densities with degeneration only on $\gamma$, while ${\rm Dens}_{\cup\gamma}(M)$ stands for such densities with degenerations on sets diffeomorphic to $\gamma$.

The following result extends Theorem~\ref{thm:no-zeros} to the case of wave functions with zeros.

\begin{theorem}\label{thm:zeros}
Let $(\nu,\varrho) \in \mathfrak{dc}^*$ where $\varrho\in {\rm Dens}_\gamma(M)$ is a  Morse-Bott density vanishing on $\gamma\subset M$, 
and $\mathcal{O}_{(\nu, \varrho)}$ is the  corresponding coadjoint orbit. 
Let $\alpha$ be the 1-form on $M\backslash\gamma$ defined by $\alpha\otimes\varrho = \nu$.
Then the following holds:

\begin{enumerate}[label=(\alph*),leftmargin=*]
\item The coadjoint orbit $\mathcal{O}_{(\nu, \varrho)}$ belongs to the image of  wave functions  $\Psi_{\cup\gamma}$ under the momentum map $\Mom$ if and only if
the cohomology class of $\alpha$ fulfills $[\alpha]\in H^1(M\backslash \gamma,2\pi\mathbb{Z})$ and the periods of $\alpha$ around components of 
     $\gamma$ are equal to $\pm 2\pi$ depending on the components' orientations.
If $\psi_j, \, j=1,2,$ belong to the same connected component $\Psi_{\cup\gamma}^{\it cc}$ of $\Psi_{\cup\gamma}$ and
$\varrho_1, \varrho_2\in {\rm Dens}_\gamma(M)$, then the difference  $\alpha_1-\alpha_2$ is an exact (smooth) 1-form on~$M$.

\item Those orbits are in 1-1 correspondence with the 
connected components of the space $\Psi_{\cup\gamma}$. Moreover, two different 1-forms $\alpha_1$ and $\alpha_2$ 
with densities $\varrho_1, \varrho_2\in {\rm Dens}_\gamma(M)$
corresponding to images $\Mom(\psi_j)=(\nu_j,\varrho_j)$ for wave functions $\psi_j, \, j=1,2,$ belonging to different connected components 
of $\Psi_{\cup\gamma}$ differ by an element of $H^1(M,2\pi\mathbb{Z})$, i.e., $[\alpha_1-\alpha_2]\in H^1(M,2\pi\mathbb{Z})$. (In particular, $\alpha_1$ and $\alpha_2$ have the same singularities on $\gamma$.)

\item Each such orbit $\mathcal{O}_{(\nu,\varrho)}$ can be regarded as a fibration over the space ${\Dens}_{\cup\gamma}(M) $ of normalized Morse-Bott densities $\varrho$ on $M$ with the critical set diffeomorphic to $\gamma$.


\item The map $\Mom\colon \Psi_{\cup\gamma}^{\it cc} \to \mathcal{O}_{(\nu, \varrho)}$ 
is an $S^1$-bundle over the orbit.
This is a prequantum bundle for each coadjoint orbit $\mathcal{O}_{(\nu, \varrho)}$ in the image of 
$\Mom(\Psi_{\cup\gamma}^{\it cc})$.

\item The 1-form $\lambda$ 
defined at $\psi$ by 
\[
    \lambda_{\psi}[\dot\psi] = \mathrm{Im}\int_M \overline{\hd}\dot\hd
\]
can be regarded as a Liouville 1-form on this prequantum bundle;

\item The map $\Mom\colon \Psi_{\cup\gamma}^{\it cc} \to \mathcal{O}_{(\nu, \varrho)}$ is a Riemannian submersion of the $L^2$ metric on $C^\infty(M,\mathbb{C})$ to an induced metric on the orbits $\mathcal{O}_{(\nu, \varrho)}$.
\end{enumerate}
\end{theorem}

\begin{proof}[Proof of Theorem~\ref{thm:zeros}]

For the most part the proof repeats the steps of Theorem~\ref{thm:no-zeros} applied to $M\backslash \gamma$ instead of $M$. 
Note that while now the 1-form $\alpha$ has an integer class in this complement,  $[\alpha]\in H^1(M\backslash \gamma,2\pi\mathbb{Z})$,
its integrals over simple loops around components of $\gamma$ are   fully prescribed by the orientation of those components.
Namely, Lemma~\ref{lem:integral-homology} is to be adjusted as follows.

\begin{lemma}\label{lem:integral-homology-zeros}
     Given $\psi \in \WFgamma\subset C^\infty(M,\mathbb{C})$ with simple zeros on $\gamma$, the cohomology class $[\alpha]$ of the closed 1-form $\alpha$ belongs to $H^1(M\backslash \gamma,2\pi\mathbb{Z})$. Moreover, the periods of $\alpha$ over simple contours around components of 
     $\gamma$ are equal to $\pm 2\pi$ depending on the component's and contour's mutual orientations.
\end{lemma}

\begin{proof}
Just like in the proof of Lemma~\ref{lem:integral-homology} we fix an arbitrary $x_0 \in M\backslash\gamma$.
  Define  a multivalued function $\theta$ on $M\backslash \gamma$ by setting $\theta(x) = 2\int_{\ell}\alpha$ where $\ell$ is a curve such that $\ell(0) = x_0$ and $\ell(1) = x$. Now the integral depends only on the homotopy class of $\ell$ from $x_0$ to $x$ in $M\backslash\gamma$. If $\ell$ is a simple contour around $\gamma$, where $\psi$ has a simple zero, then $\theta/2$ has an increment $\pm 2\pi$.
    Since any two 1-forms $\alpha$ and $\tilde \alpha$
    have the same singularity on $\gamma$ (see Proposition~\ref{prop:trans} and Theorem \ref{thm:sing-alpha}), their increments on  
    such loops $C$ coincide.
    The rest of the proof  follows that of Lemma~\ref{lem:integral-homology}. Namely, the multivaluedness of $\theta$ 
    is defined by periods of $\alpha$:
    one has $\int_\ell \alpha - \int_{\tilde \ell}\alpha = \int_{\ell-\tilde \ell} \alpha$.
    Hence, a complex-valued function $\psi$ defined by $\psi =\sqrt{\rho \mathrm e^{\mathrm i \theta}}$ on $M\backslash \gamma$ 
  is univalued if and only if the periods of $\alpha$ belong to $2\pi \mathbb Z$, i.e., $[\alpha] \in H^1(M\backslash \gamma,2\pi\mathbb{Z})$.
\end{proof}

To complete the proof of Theorem~\ref{thm:zeros}, for (a)-(b) we note that the above constraint on the periods of $\alpha$ around $\gamma$ depending only on $\gamma$ itself, implies that different 1-forms $\alpha_1$
and $\alpha_2$ have the same periods, and hence their difference $\alpha_1-\alpha_1$ ``does not notice" $\gamma$ and defines a cohomology class
in $H^1(M, 2\pi\mathbb{Z})$, rather than in $H^1(M\backslash \gamma ,2\pi\mathbb{Z})$.
The item (c) is essentially a reformulation of Theorem~\ref{thm:KV25}. 
The $S^1$-fibers of the projection $\Mom$ in (d) are guaranteed by Lemma~\ref{lem:preimage}.
For (e) we notice that $\psi$ has no singularity on its zero set $\gamma$, so the integral extends from $M\backslash \gamma$ to $M$.
The statement (f) emphasizes that the $L^2$ metric on normalized wave functions in $S^\infty_1$ is $S^1$-invariant and descends to the corresponding coadjoint orbits,  thus providing  a Riemannian submersion.
\end{proof}

\begin{remark}\label{rem:switching}
This provides an alternative point of view on  the Wallstrom controversy \cite{Wa1994, Wa1994b}.
Indeed, only special coadjoint orbits 
on the hydrodynamical side are covered by the Madelung momentum map from the space of wave functions. These orbits are related to various types of zeros and connected components of the wave functions.

It is interesting to see how this hydrodynamical view is consistent with the Schr\"odinger  dynamics. 
In general, the symplectic space $C^\infty(M,\mathbb{C})$ of smooth wave functions on $M$ is projected into the Lie-Poisson space
$\mathfrak{dc}^*$, and its image contains a set of different orbits, i.e., the image is a union of symplectic leaves.
An initially nonvanishing  wave function $\psi_t$ satisfying the (nonlinear) Schr\"odinger equation 
\eqref{eq:schrodinger} 
may develop zeros during the evolution. For instance, suppose that some Schr\"odinger solution exists for 
$t\in [0,1]$, it has no zeros at $t=0$, and it first assumes zeros at a critical time $t=t_*$. 
As long as it remains nowhere vanishing on $M$ for $t\in [0, t_*)$ its image $\Mom(\psi_t)$ satisfies the barotropic system \eqref{eq:madelung_eq} for a quantum fluid 
on the hydrodynamical side. It belongs to one and the same coadjoint orbit 
$\mathcal O_{(\nu_0, \varrho_0)}\in \mathfrak{dc}^*$ with a nonvanishing density $\varrho_t$, as it is a Hamiltonian equation on the orbit, while the Madelung map is a symplectomorphism of the appropriate spaces.
(Note that all $\varrho_t$ for $t\in [0, t_*)$ are diffeomorphic to $\varrho_0$ by Moser's theorem.)
Assuming that for $t_*<t\le 1$ the zero set $\gamma_t$ of $\psi_t$ is noncritical, the image $\Mom(\psi_t)$  belongs to a different coadjoint orbit 
$\mathcal O_{(\nu_1, \varrho_1)}\in \mathfrak{dc}^*$ with $\varrho_1$ having zeros on  $\gamma=\psi_1^{-1}(0)$. It remains on this orbit as long as $\gamma_t$
does not change the topology, see Figure~\ref{fig:projection}. 

\begin{figure}
    \centering
    \includegraphics[scale=0.8]{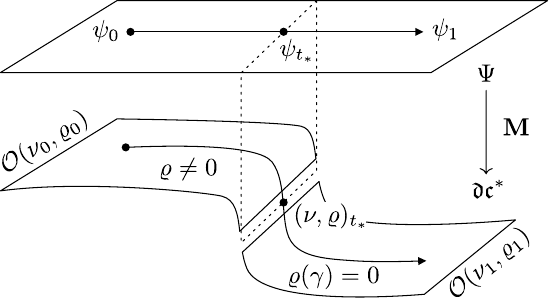}
    \caption{For a wave function $\psi_t$ satisfying the Schr\"odinger equation and
 developing zeros during the evolution, its $\Mom$-projection ``jumps" from one coadjoint orbit $\mathcal O_{(\nu_0, \varrho_0)}\in \mathfrak{dc}^*$ to another orbit  $\mathcal O_{(\nu_1, \varrho_1)}$ by switching them at the  time $t=t_*$.
 }\label{fig:projection}
\end{figure}

How the trajectory jumped from one coadjoint orbit to another? This happened when 
the wave function $\psi_{t}$ {\it had 0 as a critical value} at $t=t_*$. As $t\to t_*^-$ the corresponding density $\varrho_t$ was approaching the boundary of the density space $\Dens(M)$, and hence the boundary of the orbit $\mathcal O_{(\nu_0, \varrho_0)}\in \mathfrak{dc}^*$. After the critical moment $t=t_*$, the trajectory on the hydrodynamical side already
belongs to another orbit $\mathcal O_{(\nu_1, \varrho_1)}\in \mathfrak{dc}^*$, where the densities $\varrho_t$ for $t_*<t\le 1$ are Morse-Bott and diffeomorphic  to one and the same density $\varrho_1$
with degeneration along $\gamma$. 
This way, a trajectory $\psi_t$ in the symplectic space $C^\infty(M,\mathbb{C})$ is projected by the Poisson map $\Mom$ to a trajectory in the Poisson space $\mathfrak{dc}^*$ 
switching between symplectic leaves. Furthermore, since the trajectory is Hamiltonian on $\mathfrak{dc}^*$, it must approach the boundary with infinite speed in finite time $t_*$, since otherwise it would remain on the same leaf.
\end{remark}

\section{Reminder: the complex projective space as a toric variety}

Recall the classical setting of the torus action on the complex projective space.
This will later serve as a finite-dimensional model of the space of wave functions.

\subsection{The torus momentum map}
Let $z=(z_0, z_1,\dots,z_n)\in \mathbb{C}^{n+1}$ be  coordinates in the complex vector space
and consider the torus $\mathbb{T}^{n+1}$-action:
$z_j \to  t_j z_j$. Here $\mathbb{T}^{n+1}=(S^1)^{n+1}
=\{ (t_0, \dots, t_n)\in \mathbb C^{n+1} \mid |t_j|=1, \, j=0, \dots, n\}$.
The action is Hamiltonian for the symplectic structure $\mathrm i \, \mathrm d{z}\wedge \mathrm d\bar z $  on $\mathbb{C}^{n+1}$, while its 
 momentum map $\Mom_{\mathbb{T}}\colon \mathbb{C}^{n+1}\to \mathbb{R}^{n+1}$ has the form:
$$
\Mom_{\mathbb T}(z) = (|z_0|^2, |z_1|^2,\dots, |z_n|^2)\,. 
$$

The fiber of this map represents the configuration of phases.

The corresponding projective space $\mathbb{CP}^n$ (with homogeneous coordinates 
$[z]=(z_0:\dots :z_n)$) is equipped with the {\it Fubini--Study form} $\omega_{FS}$.  
It is a toric manifold with the natural action 
of the $n$-dimensional torus
$\mathbb T^n=\mathbb T^{n+1}/S^1$. This action $[z]:=[z_0:z_1:\dots:z_n]\mapsto [z_0:t_1 z_1:\dots:t_n z_n]$ is Hamiltonian, and abusing notations
we denote the corresponding momentum map by the same symbol:
$\Mom_{\mathbb{T}}: \mathbb{CP}^{n}\to \mathfrak{t}^*\simeq\mathbb{R}^{n}$, 
$$
\Mom_{\mathbb T}([z]) = \frac{1}{\sum_{j=0}^n |z_j|^2} \left( |z_1|^2, |z_2|^2, \dots, |z_n|^2 \right)\,,
$$
where for $[z]\in \mathbb{CP}^{n}$. Now the fibers represent the relative configuration of phases.
The image of the momentum map of any toric manifold is a convex polytope by the Atiyah--Guillemin--Sternberg theorem. For $\mathbb{CP}^n$, it is the  standard $n$-simplex 
$$
\Delta_n = \left\{ (x_1, \dots, x_n) \in \mathbb{R}^n \mid x_i \geq 0, \, \sum_{i=1}^n x_i \leq 1 \right\}\,.
$$
One can also describe this simplex in a more symmetric form (and more convenient for the infinite-dimensional generalization) as 
$$\Delta_n = \left\{ (x_0, x_1, \dots, x_n) \in \mathbb{R}^{n+1} \mid x_i \geq 0, \, \sum_{i=0}^n x_i =1 \right\}\,.
$$

\subsection{The $U(n+1)$ action and momentum map}
The larger group $G = U(n+1)$ acts transitively on $\mathbb{CP}^n$ via $A: [z] \mapsto [Az]$. This action is Hamiltonian, and its momentum map $\Mom_G$ takes values in the dual Lie algebra $\mathfrak{u}(n+1)^*$. Identify the latter with the space of Hermitian matrices.
One can show that the momentum map $\Mom_G: \mathbb{CP}^n \to \mathfrak{u}(n+1)^*$ is given by the rank-one projection matrix:
$$
\Mom_G([z]) = \frac{z  z^\dagger }{\|z\|^2} 
$$
for $z=(z_0, z_1,\dots,z_n)$.
Explicitly, the entry in the $j$-th row and $k$-th column is:
$\Mom_G([z])_{jk} =  {z_j \bar{z}_k}/{\sum_{l=0}^n |z_l|^2} $.
The image of  $\mathbb{CP}^n$ under the map $\Mom_G$ turns out to be a single coadjoint orbit of $U(n+1)$ in $\mathfrak{u}(n+1)^*$, for instance, the orbit of the matrix $\text{diag}(1, 0, \dots, 0)$.

In this setting the torus momentum map $\Mom_{\mathbb{T}}$ is obtained by taking the diagonal entries of the matrix $\Mom_G$ (while ignoring the constant $z_0$ term). Alternatively, one can consider the 
$SU(n+1)$-action, rather than $U(n+1)$ action, on $\mathbb{CP}^n$. This action is effective and its
maximal torus is already $\mathbb{T}^{n}$. The corresponding momentum map $\tilde{\Mom}_G:= \Mom_G-\frac{1}{n+1}I$ is the traceless part of $\Mom_G$ (which is equivalent to ignoring the constant $z_0$ term in another chart).  However, the $U(n+1)$ action is often more convenient for the infinite-dimensional analogy.

\begin{remark}
  To   summarize, the action of the torus $\mathbb{T}^n$ on $\mathbb{CP}^n$ is Hamiltonian and effective,
  while the action of the unitary group $U(n+1)$ is Hamiltonian and transitive. 
  The images of the corresponding momentum maps are the $n$-simplex in $\mathbb{R}^n$ in the toric case, and the full coadjoint orbit $\mathfrak{u}(n+1)^*$ diffeomorphic (and sympletomorphic) to $\mathbb{CP}^n$ in the unitary case. Notes that the simplex's faces and edges correspond to toric orbits of smaller dimensions, where one or several coordinates vanish, $z_j=0$. As we discuss below,  in the continuous case, the faces correspond to wave functions with zero sets such as $\gamma$ .
\end{remark}

\subsection{The Hopf fibration and prequantization}
The sphere $S^{2n+1}$ serves as the {\it prequantum bundle} for $(\mathbb{CP}^n, \omega_{FS})$. This relationship arises from the symplectic reduction of $\mathbb{C}^{n+1}$ by the diagonal $S^1$-action.

Namely, consider the unit sphere $S^{2n+1} = \{ z \in \mathbb{C}^{n+1} \mid \sum |z_j|^2 = 1 \}$. The principal $S^1$-bundle $\pi: S^{2n+1} \to \mathbb{CP}^n$ is equipped with a connection 1-form $\lambda \in \Omega^1(S^{2n+1})$ defined by:
\[ \lambda = \mathrm{Im}(\bar z \mathrm{d}z) =  \mathrm{Im}\Big( \sum_{j=0}^n \bar{z}_j \mathrm{d}z_j\Big) \]
This is the restriction of the canonical Liouville 1-form on $\mathbb{C}^{n+1}$. 
It satisfies the prequantization condition:
\[ \mathrm{d}\lambda = \pi^* \omega_{FS} \]
where $\omega_{FS}$ is the Fubini--Study symplectic form on $\mathbb{CP}^n$, which in turn  arises from the symplectic form $\mathrm i\,\mathrm{d} z\wedge \mathrm{d}\bar z $ on $\mathbb{C}^{n+1}$.

\subsection{A comment on weighted torus actions and Delzant polytopes}\label{sec:general_torus}
Consider a more general action of the $n$-dimensional torus $\mathbb{T}^n$. Namely,  assign an integer weight vector $\mathbf{w}_j \in \mathbb{Z}^n$ to each homogeneous coordinate $z_j, j=1, \dots ,n$. (One  can set $\mathbf{w}_0 = \mathbf{0}$ without loss of generality.) Then for an element $t = (t_1, \dots, t_n) \in \mathbb{T}^n$, the action on $\mathbb{CP}^n$ is 
$$
[z_0 : z_1 : \dots : z_n]\to [ z_0 : t^{\mathbf{w}_1} z_1 : \dots : t^{\mathbf{w}_n} z_n]\,,
$$ 
where $t^{\mathbf{w}_j} = \prod_{k=1}^n t_k^{w_{jk}}$.

One can see that this action is still Hamiltonian for the Fubini--Study symplectic structure $\omega_{FS}$ on $\mathbb{CP}^n$. 
The corresponding  momentum map $\Mom_{\mathbf w}: \mathbb{CP}^n \to \mathbb{R}^n$ is:
$$
\Mom_{\mathbf w}([z_0 : \dots : z_n]) =  \frac{\sum_{j=1}^n \mathbf{w}_j |z_j|^2}{\sum_{j=0}^n |z_j|^2}\,.
$$
The momentum image $\Delta^{\mathbf w}_n = \Mom_{\mathbf w}(\mathbb{CP}^n)$ is the convex hull of the images of the fixed points. The fixed points of the action are the coordinate points $P_0 = [1:0:\dots:0], \dots, P_n = [0:\dots:0:1]$. Their images are:
$$
\Mom_{\mathbf w}(P_0) = \mathbf{0}, \quad \Mom_{\mathbf w}(P_j) =  \mathbf{w}_j \text{ for } j = 1, \dots, n\,.
$$
The resulting momentum image is the simplex: $\Delta^{\mathbf w}_n = \text{conv}\{ \mathbf{0},  \mathbf{w}_1,  \mathbf{w}_2, \dots,  \mathbf{w}_n \}\,.$
\medskip

Recall that a convex polytope $\Delta \subset \mathbb{R}^n$ is called a \textit{Delzant polytope} 
if it satisfies the following three conditions  ensuring that the associated toric variety is a smooth compact symplectic manifold. Namely, the polytope is supposed to be \textit{simple} 
(i.e.,  there are exactly $n$ edges meeting at each vertex),  \textit{rational} (i.e., the edges 
meeting at each vertex have primitive integral vectors  $u_j\in \mathbb{Z}^n$ as their directions) 
and \textit{smooth} (i.e., satisfying the Delzant condition). 
The latter means that at each vertex, the set of primitive vectors $\{u_1, \dots, u_n\}$ forms a basis for the lattice $\mathbb{Z}^n$, or equivalently, $
\det(u_1, u_2, \dots, u_n) = \pm 1$.

The standard simplex $ \Delta_n=\Mom_{\mathbb T}(\mathbb{CP}^n)$ is Delzant. The 
simplex $\Delta^{\mathbf w}_n=\Mom_{\mathbf w}(\mathbb{CP}^n)$ discussed above remains a Delzant polytope if and only if the matrix $W = [\mathbf{w}_1 | \dots | \mathbf{w}_n]$ formed by the vectors $\mathbf{w}_j$ satisfies $\left|\det(W)\right| = 1$. If $\left|\det(W)\right| > 1$, the resulting space is a \textit{weighted projective space}, which contains orbifold singularities at the fixed points, and the polytope is called a {\it rational orbitope}.
(The weighted projective space is obtained  via symplectic reduction from $\mathbb C^{n+1}$ by a weighted $\mathbb C^*$-action.)

\begin{remark}
  There are also generalizations to manifolds $(\mathbb{CP}^{n},\mathcal{O}(k))$ representing the complex projective space with the $k$-th power of the hyperplane line bundle. Its momentum image is an $n$-dimensional simplex $\Delta_n$ scaled by $k$. 

In geometric quantization, the sections of the associated complex line bundle $L \to \mathbb{CP}^n$ 
are equivalent to $S^1$-equivariant functions on $S^{2n+1}$.  For the $k$-th power of the prequantum 
bundle (corresponding to $k\omega_{FS}$), the holomorphic sections are identified with homogeneous 
polynomials of degree $k$ on $\mathbb{C}^{n+1}$.
\end{remark}



\section{The finite-dimensional model of the Madelung transform}\label{sect:finiteMadelung}
The momentum map for $\mathbb{CP}^n$ can be interpreted as a finite-dimensional Madelung transform. 
The Madelung map decomposes a wave function $\psi=\sqrt{\rho \mathrm e^{\mathrm i\theta}}$ into its density $\rho$ and phase velocity $v=\nabla \theta$.

\subsection{Discrete densities and phases}
In a discrete (or finite-dimensional) setting one considers the scalar-valued wave function $\psi$ 
defined on several points, $M=\{q_0, \dots, q_n\}$ and $z_i=\psi(q_i)$. Alternatively, one can consider 
a vector valued wave function
$\vec{\psi}(q):=(\psi_0(q), \dots, \psi_n(q))=(z_0, \dots, z_n)\in \mathbb{C}^{n+1}$ defined 
over a single point $M=\{q\}$. 

Express  the coordinates in polar form $z_j = \sqrt{\rho_j \mathrm e^{\mathrm i \theta_j}}$. 
The torus $\mathbb{T}^{n+1}$-action, as before, is the coordinate-wise $S^1$-action:
$z_j \to t_j z_j$ with $t_j\in S^1$, while its 
 momentum map $\Mom_{\mathbb{T}}: \mathbb{C}^{n+1}\to \mathbb{R}^{n+1}$ has the form:
$$
\Mom_{\mathbb T}(z) =   (\rho_0, \rho_1, \dots, \rho_n)\,. 
$$

Note that this component-wise torus action exactly corresponds to the action by functions only, i.e., the action of the subgroup $\mathbb T^\infty\simeq\{\mathrm{id}\}\times C^\infty(M, S^1)\subset \DC$ in the continuous 
version. 
If $({\rm id}, a)\in \DC$ is an element of the abelian subgroup of functions acting on 
half-densities, its action is $ ({\rm id},a)\cdot \hd  = \hd\operatorname{exp}(-\mathrm i a /2)$ which 
boils down to the change of phases of $z_i=\psi(q_i)$ at the points $q_i$.
The corresponding momentum map is just the projection of Fusca's momentum map $\Mom$ to the second component and it is thus given by $\psi \mapsto \bar\hd\hd = \varrho$ in line with the discrete momentum map $\Mom_{\mathbb T}$ above.

Next, recall that the projective space $\mathbb{CP}^n$ with the torus $\mathbb T^n$-action  
has the momentum map 
$\Mom_{\mathbb{T}}: \mathbb{CP}^{n}\to \mathfrak{t}^*\simeq\mathbb{R}^{n}$, 
$$
\Mom_{\mathbb T}([z]) = \frac{1}{\sum_{j=0}^n\rho_j}(\rho_1, \rho_2, \dots, \rho_n)\,,
$$
where for $[z]\in \mathbb{CP}^{n}$ and $\rho_j=|z_j|^2$.
The image of $\mathbb{CP}^{n}$ under the momentum map $\Mom_{\mathbb{T}}$  is the 
standard $n$-simplex $\Delta_n$ in the $\rho$-coordinates: 
$$
\Delta_n = \left\{ (\rho_1, \dots, \rho_n) \in \mathbb{R}^n \mid \rho_i \geq 0, \, \sum_{i=1}^n \rho_i \leq 1 \right\}
$$
$$~ ~ \simeq\left\{ (\rho_0, \dots, \rho_n) \in \mathbb{R}^{n+1} \mid \rho_i \geq 0, \, \sum_{i=0}^n \rho_i =1 \right\}\,.
$$

One can see that the latter description of the convex set $\Delta_n $ corresponds
to the space $ \mathrm{ Dens}(M)=\{\varrho\in \Omega^n(M)~|~\varrho \ge 0, \int_M \varrho=1\}$ of densities on $M$ in the infinite-dimensional setting. We are particularly interested 
in the sets $\gamma=\{x\in M~|~\varrho(x)=0\}$, which correspond to faces of $\Delta_n$.
This way the (pointwise) phase rotation by the abelian group $ C^\infty(M,S^1)$ on the space  $C^\infty(M,\mathbb C)$ of wave functions $\psi$ can be regarded as the action of the infinite-dimensional torus $\mathbb{T}^\infty := C^\infty(M,S^1)$.

\medskip

\subsection{Relation to the Liouville form}
Now we confine to the unit 
sphere $\{\sum |z_i|^2=1\}=S^{2n+1}\subset \mathbb{C}^{n+1}$, corresponding to $\sum_{j=0}^n \rho_j = 1$.

\begin{proposition}
    The prequantization 1-form $\lambda$ in the Hopf fibration $S^{2n+1} \to \mathbb{CP}^n$  can be written in these coordinates as the Liouville form
$ \lambda = \frac{1}{2}\sum_{j=0}^n \rho_j \mathrm{d}\theta_j \,.$
\end{proposition}

Thus 
the Hopf fibration is the geometric realization of ``modding out'' the global phase while keeping the information of relative phases and local densities.

\begin{proof}
Represent the vector wavefunction $\vec{\psi} =(\psi_0, \dots, \psi_n)\in \mathbb{C}^{n+1}$  in polar form $\psi_j = \sqrt{\rho_j e^{\mathrm{i}\theta_j}}$, then the Liouville 1-form ${\rm Im} \sum\bar{\psi_j}\, \mathrm{d}\psi_j$ in
$\mathbb{C}^{n+1}$  expands as:
$$
\sum_{j=0}^n \bar{\psi}_j \mathrm{d}\psi_j = \sum_{j=0}^n \frac{1}{2}\left(  \mathrm{d}\rho_j + \mathrm{i} \rho_j \mathrm{d}\theta_j \right) \,.
$$
On the unit sphere $S^{2n+1}$, the total probability is conserved ($\sum \rho_j = 1$), which implies that $\sum \mathrm{d}\rho_j = 0$. 
Thus, the real part vanishes, and we obtain
$$ {\rm Im} \sum\bar{\psi_j} \mathrm{d}\psi_j = \frac12\sum_{j=0}^n \rho_j \mathrm{d}\theta_j \,,$$
which is the 1-form $\lambda$.
\end{proof}

The infinite-dimensional version of the Liouville 1-form $\lambda$ is the 1-form $\lambda_\psi$ described in Theorems \ref{thm:no-zeros} and \ref{thm:zeros}.

\begin{remark}
    The Liouville 1-form $\lambda$ can be regarded as the Berry connection. For an evolution along a closed path $\ell$ in $\mathbb{CP}^n$, the geometric (Berry) phase  $\theta_{B}$ is the path integral of this 1-form:
$$ e^{\mathrm{i} \theta_B} = \exp \left( \mathrm{i} \oint_{\ell} \lambda \right) = \exp \left( \oint_{\ell} \bar{\psi}\, \mathrm{d}\psi \right) \,.
$$
\end{remark}

\begin{remark}
 So far in the context of Madelung transform we confined ourselves to the case of \emph{simple zeros} 
 of wave functions. For instance, assume that on a transversal to its zero set $\gamma$ a wave function locally has the form $\psi(z)=z$ with  a simple zero on $\gamma:=\{z=0\}$. 
 The corresponding Berry connection for a simple loop $\ell=\{|z|=1\}$ around $\{z=0\}$ gives
 $\oint_{\ell} \bar{z}\, \mathrm{d}z=\int_0^{4\pi} e^{-\mathrm{i}\theta/2}\mathrm{d}e^{i\theta/2}=2\pi \mathrm{i}$.

 This can be compared with a wave function having  zero of higher multiplicity $m\ge 2$. For instance, for $\psi(z)=z^m$ one has 
 $\oint_{\ell} \bar{z}^m\, \mathrm{d}z^m=\int_0^{4\pi} e^{-\mathrm{i}m\theta/2}\mathrm{d}e^{\mathrm{i}m\theta/2}=2m\pi \mathrm{i}$.
This suggests that \emph{multiple zeros} of  wave functions can be regarded as analogues of more general torus actions (weighted actions with faster rotation of the corresponding coordinates), as discussed in Section \ref{sec:general_torus}. The corresponding zero sets $\gamma$
are not necessarily smooth in that case, since now they are critical levels of wave functions, while the Morse-Bott theorem for volume forms is not directly applicable. 
However, the analogy seems to extend naturally to the 
case of higher multiplicity for zeros of wave functions, and one can still regard the case of 
$\gamma$ with higher multiplicity as analogues of the faces of the momentum polytope with different slopes.
Also, it is well-known that if some weights are collinear, the image of the torus action on $\mathbb{CP}^n$ is a  lower-dimensional polytope. This corresponds to wave functions $\psi$ having more degenerate sets of zeros $\gamma$, for instance, of co-dimension less than 2. 

This opens a way to study analogues of general 
 Delzant polytopes, as well as their degeneracies,  in the infinite-dimensional context of the Madelung transform.
\end{remark}


\subsection{The unitary group action and symplectomorphisms}

Now consider the unitary group $G=U(n+1)$ acting on $[\vec{\psi}]=(\psi_0:\dots:\psi_n)\in \mathbb{CP}^n$. 
As we discussed above, in this case $\mathbb{CP}^n$ is symplectomorphic to its
 image under the corresponding momentum map $\Mom_G$, and it is  
 a single coadjoint orbit of in $\mathfrak{u}(n+1)^*$.
This is exactly a finite-dimensional analogue of the symplectomorphism statement about Madelung transform
given in Theorem \ref{thm:zeros}, where the abelian subgroup of functions is an analogue of the torus action, while the whole semi-direct product group 
$\DC = \mathrm{Diff}(M)\ltimes C^\infty(M, S^1)$ corresponds to the full 
unitary group $G=U(n+1)$, and the momentum map is a symplectomorphism.

\begin{remark}
While the group $U(n+1)$ could be regarded as an analogue of the group of all unitary maps of $L^2(M,\mathbb{C})\to L^2(M,\mathbb{C})$ in the continuous  setting, it is convenient to consider 
 in  infinite dimensions its much smaller subgroup $\DC=\mathrm{Diff}(M)\ltimes C^\infty(M, S^1)$. Indeed, as we discussed, its unitary  action on $C^\infty(M,\mathbb{C})$ is transitive on  the spaces of wave functions with a given type of zero sets $\gamma$, and this transitivity is exactly what is needed for the identification of those   spaces with the corresponding coadjoint orbits in $\mathfrak{dc}^*$.

This also shed some light on the origin of the group $\DC$. Its Lie algebra $\mathfrak{dc} = \mathfrak{X}(M)\ltimes C^\infty(M,\mathbb{R})$ can be understood as a Lie subalgebra in $\mathfrak{ham}(T^*M)$, since any diffeomorphism on $M$
gives rise to a Hamiltonian diffeomorphism of it cotangent bundle $T^*M$, while the factor $C^\infty(M,\mathbb{R})$ corresponds to (also Hamiltonian) shifts of the fibers. On the other hand, the group 
$\mathrm{Ham}(T^*M)$ acting on functions on $T^*M$ can be regarded as a natural analogue of the finite-dimensional unitary group. 
\end{remark}

\subsection{Fisher--Rao metric and $\mathbb{CP}^n$}
Furthermore, the momentum map is not only symplectic 
but it is also an isometry (and hence K\"{a}hler)
in both finite- and infinite-dimensional settings.

In finite dimensions, 
the Fubini--Study metric on $\mathbb{CP}^n$ is obtained by the projection of the round metric on $S^{2n+1}$, which in turn, is simply the restriction of the Euclidean metric $\sum |\mathrm{d}z_j|^2$ to the sphere. 
Given the principal $S^1$-bundle $\pi: S^{2n+1} \to \mathbb{CP}^n$ and the connection 1-form $\lambda$, the total metric on the sphere $g_{S^{2n+1}}$ can be reconstructed from the base metric via:
$g_{S^{2n+1}} = \pi^* g_{FS} + 2 \lambda \otimes \lambda $. 
(This is an example of the Sasaki lift in information geometry.)     

\begin{proposition}
In polar coordinates, 
the Fubini--Study metric on $\mathbb{CP}^n$ decomposes into a classical density part and a quantum phase part:
\[ 
    g^{\it FS} = \underbrace{\frac{1}{2} \sum_{j=0}^n \frac{\mathrm{d}\rho_j^2}{\rho_j}}_{\text{Fisher--Rao}} +\underbrace{\frac{1}{2}\big(\sum_{j=0}^n \rho_k \mathrm{d}\theta_k^2\big) - 2\, \lambda\otimes\lambda}_{\text{Sasaki extension}}
\]
where $\lambda = \frac{1}{2}\sum_j \rho_j \mathrm{d}\theta_j$ is the Berry connection.   
\end{proposition}
\begin{proof}
    We have $z_k = \sqrt{\rho_k \mathrm{e}^{\mathrm i \theta_k}}$, so
    \begin{equation*}
        \mathrm{d}z_k = \frac{\mathrm e^{\mathrm i \theta_k}\mathrm{d} \rho_k + \mathrm i z_k^2 \mathrm{d}\theta_k}{2 z_k} = \frac{\mathrm e^{\mathrm i\theta_k/2}\mathrm{d}\rho_k}{2\sqrt{\rho_k}} + \frac{\mathrm i z_k \mathrm{d}\theta_k}{2}
    \end{equation*}
    and 
    \begin{equation*}
        \mathrm{d}\bar{z}_k = \frac{\mathrm e^{-\mathrm i \theta_k}\mathrm{d} \rho_k - \mathrm i \bar{z}_k^2 \mathrm{d}\theta_k}{2 \bar{z}_k} = \frac{\mathrm e^{-\mathrm i\theta_k/2}\mathrm{d}\rho_k}{2\sqrt{\rho_k}} - \frac{\mathrm i \bar{z}_k \mathrm{d}\theta_k}{2}\,.
    \end{equation*}
    Thus,
    \begin{equation*}
        \mathrm{d}z_k\otimes \mathrm{d} \bar{z}_k + \mathrm{d}\bar{z}_k\otimes \mathrm{d} z_k = 
        \frac{\mathrm{d}\rho_k \otimes \mathrm{d}\rho_k}{2 \rho_k} + \frac{\rho_k 
        \mathrm{d} \theta_k \otimes \mathrm{d}\theta_k}{2}.
    \end{equation*}
    
    The total metric tensor should vanish for the vertical vector $v = c \sum \frac{\partial}{\partial \theta_k}$.
    We have that
    \begin{equation*}
        \big(\frac{1}{2}\sum_k \rho_k \mathrm{d}\theta_k \otimes \mathrm{d}\theta_k \big) (v,v) = \frac{1}{2}c^2 \sum_k \rho_k = \frac{1}{2} c^2.
    \end{equation*}
    Since $\lambda = \frac{1}{2}\sum_k \rho_k \mathrm d\theta_k$ we also have
    $
        \lambda\otimes\lambda (v,v) = \frac{1}{4} c^2 \big(\sum_k \rho_k \big)^2 = \frac{1}{4} c^2.
    $
    Thus, the metric is
    \begin{equation*}
        \frac{1}{2}\big( \sum_k \frac{\mathrm{d}\rho_k\otimes \mathrm{d}\rho_k}{\rho_k}\big) +\frac{1}{2}\big(\sum_k \rho_k \mathrm{d}\theta_k\otimes \mathrm{d}\theta_k\big) - 2\lambda\otimes\lambda .
    \end{equation*}
    Notice, as expected, that a vector $w$ is horizontal if $\iota_w \lambda = 0$.
\end{proof}

\begin{remark}
    The Fubini--Study metric $g^{\it FS}$ can also be written
    \[ 
        g^{\it FS} = \frac{1}{2} \sum_{j=0}^n \frac{\mathrm{d}\rho_j^2}{\rho_j} + \frac{1}{2}\sum_{j=0}^n \rho_j (\mathrm{d}\theta_j - 2\lambda)^{\otimes 2},
    \]
    provided that $\sum_j \rho_j = 1$.
    Indeed, the form above can be seen as the variance form of the metric corresponding to the weighted variance of the phase differentials $\mathrm{d}\theta_j$ with respect to the probability distribution $\rho_j$, and the relation to the first form of the metric is obtained via the variance formula
    \begin{multline*}
        \frac{1}{2}\sum_{j=0}^n \rho_j (\mathrm{d}\theta_j - 2\lambda)^{\otimes 2}
        = \frac{1}{2}\sum_j \rho_j \mathrm{d}\theta_j^2 + 2\lambda^{\otimes 2} - 4\lambda^{\otimes 2} = \frac{1}{2}\sum_j \rho_j \mathrm{d}\theta_j^2 -2\lambda^{\otimes 2}.
    \end{multline*}
\end{remark}

Note that the inclusion of $\lambda$ ensures $S^1$-invariance of the metric, so that it descends to $\mathbb{CP}^n$. Here the first term is the Fisher--Rao metric, representing the information geometry of the  simplex $\Delta_n$ of classical probability distributions.  The second term stands for the phase geometry related to quantum mechanics.

\medskip

One can compare the above formula with its infinite-dimensional version.
In infinite dimensions for the case a simply-connected manifold $M$ and wave functions without zeros
this was proved in \cite{KhMiMo2019}. Namely, the corresponding space $\mathbb{P}C^\infty(M, \mathbb{C}^*)\subset\mathbb{CP}^\infty$ equipped with the Fubini--Study metric
was shown to be isometric to the cotangent bundle $T^*\Dens(M)$ of the density space on $M$, which is equipped with the {\it Fisher--Rao-Sasaki  metric}:
$$
 g^{\it FRS}_{(\varrho,[\theta])}((\dot \varrho,\dot\theta) , (\dot \varrho,\dot\theta)):=\frac{1}{2}\int_M \frac{\dot\varrho^2}{\varrho}    +\frac{1}{2}\int_M \dot\theta^2\varrho\,  - 2 \big(\underbrace{\frac{1}{2}\int_M \dot\theta \varrho}_{\text{Liouville}} \big)^2.
$$
This is an extension of the Fisher--Rao metric on $\Dens(M)$, given by the first term depending only on densities (or probability distributions) $\varrho$, to the cotangent bundle $T^*\Dens(M)$.

\begin{theorem}[\cite{KhMiMo2019}]
The Madelung transform is an  isometry (and hence a  K\"{a}hler  map) between 
the spaces $T^*\Dens(M)$ equipped with the Fisher--Rao-Sasaki  metric
and $\mathbb P\WFnonzero = \mathbb P C^{\infty}(M,\mathbb C^*)$ equipped with the  Fubini--Study metric.
\end{theorem}

Recall that the (infinite-dimesional) {\it Fubini--Study metric}  on $\mathbb{P}\WF=\mathbb P C^\infty(M,\mathbb C)$ is
$$
g^{\it FS}_\psi(\dot\psi, \dot\psi):=\frac{\langle\dot\psi,\dot\psi\rangle}{\langle\psi, \psi \rangle} - \frac{\langle\psi,\dot\psi\rangle\langle\dot\psi,\psi\rangle}{\langle\psi, \psi \rangle^2}
$$
where $\psi$ is a representative wave function in the coset $[\psi]\in \mathbb P \WF$, and 
$\langle\psi, \psi \rangle$ is the $L^2$ inner product on $C^\infty(M,\mathbb C)$.


\subsection{The comparison of finite- and infinite-dimensional convexity}\label{sect:table}

Here is an informal table indicating the analogues of the classical convexity theorem in finite dimensions and the 
infinite-dimensional convexity implied by the Madelung transform.

\begin{center}
\renewcommand{\arraystretch}{1.35}
\begin{longtable}{p{0.39\textwidth} p{0.55\textwidth}}
\toprule
\textbf{Finite-dimensional toric \newline geometry} &
\textbf{Madelung / $\DC$ Hamiltonian \newline geometry} \\
\midrule
\endhead

Unitary group
$
U(n+1)
$
acts on
$
\mathbb{CP}^n
$
&
Semidirect product
$
\DC
$
acts on
$
\mathbb{P}\WF
$
\\

Maximal torus
$
\mathbb T^n=T^{n+1}/S^1
$
&
Infinite-dimensional phase torus
$
\mathbb T^\infty=C^\infty(M,S^1)
$
acting by pointwise phase rotations
\\

Normalizer
$
N(\mathbb T^n)\subset U(n+1)
$
&
Normalizer

$N(\mathbb T^\infty)
=
{\rm Diff}(M)\ltimes C^\infty(M,S^1)
=
\DC
$
\\

Weyl group \newline 
$
W=N(\mathbb T^n)/\mathbb T^n
\simeq S_{n+1}
$
&
Weyl-type quotient
$
N(\mathbb T^\infty)/\mathbb T^\infty
\simeq {\rm Diff}(M)
$
\\

Momentum map \newline
$
\Mom_{\mathbb T}\colon\mathbb{CP}^n\to{\mathfrak t^n}^{*}\simeq \mathbb R^n
$
&
Momentum map \newline
$
\Mom_{\mathbb T^\infty}\colon\mathbb{CP}^{\infty}\to
{\mathfrak t^\infty}^{*}
\simeq \Omega^n(M)
$
\\

Moment image\newline
$
\Delta^n
=
\{x_i\ge0,\ \sum x_i=1\}
$
&
Probability polytope (after completion)\newline
$
\overline\Dens(M)
=
\{\varrho\ge0,\ \int_M\varrho=1\}
$
\\

Interior of simplex
&
Smooth positive densities
\\

Faces
$
\Delta^k\subset \Delta^n
$
&
Faces
$
F_\Gamma
=
\{\varrho\in\overline\Dens(M)\mid
\varrho(\Gamma)=0\}
$
\\

Stratification by faces
&
Stratification by support type
\\

Vertices
$
e_i
$
&
Dirac measures
$
\delta_x
$
\newline
(in the measure completion)
\\

Weyl group permutes vertices
&
{\rm Diff}(M) moves Dirac masses
$
\varphi_*\delta_x=\delta_{\varphi(x)}
$
\\

Parabolic subgroup
$
W_I
$ \newline 
preserves a face $I$
&
Subgroup
$
{\rm Diff}_\Gamma(M)
\coloneqq
\{\varphi~|~\varphi(\Gamma)=\Gamma\}
$ \newline 
preserves the face
$
F_\Gamma
$
\\

Convex combination
$
\sum_i a_i e_i
$
&
Barycentric decomposition
$
\varrho
=
\int_M \delta_x\, \rho(x)
$
\\

Vertices span all faces
&
Dirac masses generate all measures
\\

\bottomrule
\end{longtable}
\end{center}

Note that the closest structural counterpart for the Weyl group $W$, which plays a central role in the Schur-Horn-Kostant theorem, a precursor of the convexity results of Atiyah and Guillemin-Sternberg, is the action of the diffeomorphism group 
${\rm Diff}(M)$ on the ``continuous polytope" $\overline\Dens(M)$ of probability measures. By Moser's theorem, ${\rm Diff}(M)$ acts transitively on the interior consisting of smooth positive densities of fixed total mass. This differs from the intriguing approach of \cite{bloch1993schur} discussed below,
where an analogue of the Weyl group for the group of area-preserving diffeomorphisms of an annulus was a certain extension of invertible measure preserving transformations of an interval. 


\subsection{Infinite-dimensional torus actions for symplectomorphisms}\label{sect:symplectomorphisms}
The existence of infinite-dimensional analogues of
maximal tori has been studied in several papers from a different perspective, which we sketch here and compare with the one above. 
In \cite{bao1997maximal, bloch1993schur} the authors considered 
the group ${\rm Diff}_{\omega}(A) $ of area-preserving diffeomorphisms of the annulus $A = [0, 1]\times S^1$ with the standard area form $\omega$ and
showed that the set of all ``pure twist” maps of Sobolev class $H^s$ with $s > 2$
$$
\mathcal T=\{\eta_\phi~|~\eta_\phi(z,\theta)=(z, \theta+\phi(z)), \,\, \phi:[0,1]\to \mathbb R \text{~ is in ~} H^s\}
$$
share some fundamental properties with finite-dimensional maximal tori in
compact Lie groups. Namely, topologically $\mathcal T$ is  a real infinite-dimensional connected closed submanifold of the symplectomorphism group ${\rm Diff}_{\omega}(A)$ relative to the $H^s$ topology.
On the other hand, algebraically, $\mathcal T $ is a maximal abelian subgroup of 
${\rm Diff}_{\omega}(A) $, while geometrically
$\mathcal T$ is a totally geodesic and flat Riemannian submanifold of ${\rm Diff}_{\omega}(A)$
with respect to a natural right-invariant $L^2$-metric.

In \cite{bloch1993schur} the authors considered a certain completion of the group ${\rm Diff}_{\omega}(A)$ to produce a nontrivial Weyl group and proved an infinite-dimensional version of 
the Schur-Horn-Kostant convexity theorem for 
the torus $\mathcal T $-action on a typical coadjoint orbit of ${\rm Diff}_{\omega} (A)$. 
In this setting, the Lie algebra of $\mathcal T $ is identified with functions on $A = [0, 1]\times S^1$ that depend only on the first coordinate, and the role of the momentum map
$\Mom_{\mathcal T }\colon \mathfrak{X}^*_\omega (A)\to \mathfrak t^* $  is played by the projection
$
\Mom_{\mathcal T }\colon f(z, \theta) \mapsto \int_0^1 f(z, \theta)\, d\theta\,.
$

\begin{theorem}[\cite{bloch1993schur}]
Let $\Lambda \in L^2(A)$ be a bounded nonincreasing right-continuous function of $z$, 
and let $\mathcal O_\Lambda$ be the orbit   through $\Lambda$ for a suitable completion of measure-preserving transformations of $A$. Then $\Mom_{\mathcal T } (\mathcal O_\Lambda)\subset 
L^2[0, 1]$ is a weakly compact convex set. Its set of extreme
points is the orbit $W\cdot\Lambda$  of the Weyl semigroup through $\Lambda$. 
\end{theorem}

Note that the maximal torus $\mathcal T $ can be viewed as the set of all 
symplectomorphisms that preserve
the level sets of the momentum map $\mu : [0, 1]\times S^1 \to [0, 1] $ (for the $S^1$-action on annulus $A$) given by the projection
onto the first component. This interpretation of $\mathcal T $ for the group
 ${\rm Diff}_{\omega}(A)$ allowed El~Hadrami
to conjecture that, given a symplectic toric manifolds $(M^{2n}, \omega, T^n, \mu)$
with momentum map $\mu$ for the torus $T^{n}$-action on $M^{2n}$, the subgroup of symplectic ``twist maps”, equivariant symplectomorphisms,
$$
\widetilde{ \mathcal T } =\{ \eta \in {\rm Symp}_\omega (M^{2n})~|~\mu\circ\eta=\mu \}\,,
$$
behaves like $\mathcal T $ in the group ${\rm Diff}_{\omega}(A) ={\rm Symp}_\omega(A)$, and it can be viewed as an analogue of a maximal
torus in the symplectomorphism group ${\rm Symp}_\omega (M^{2n})$, see \cite{hadrami1996poisson}. 
In particular, he proved that $\widetilde{ \mathcal T }$ has the expected properties in the special cases of $M = \mathbb{CP}^1$ or $\mathbb{CP}^2$. 

\smallskip

An extension of those results related to an infinite-dimensional convexity theorem  to the general case was implemented in \cite{mousavi2026}. 
Namely, let $(M, \omega, T, \mu)$ be a $2n$-dimensional toric manifold. It turns out that 
the subgroup of all equivariant symplectomorphism, $ \widetilde{ \mathcal T }={\rm Symp}_\omega(M, T)$, can play the role
of a maximal torus in the symplectomorphism group $ {\rm Symp}_\omega(M)$.
In \cite{mousavi2026} the authors show  that $ {\rm Symp}_\omega(M, T)$  is a closed, infinite-dimensional and
path-connected submanifold of $ {\rm Symp}_\omega(M)$ which is flat and totally geodesic with
respect to any ``toric" weak Riemannian metric. Moreover, it is a maximal
abelian subgroup of $ {\rm Symp}_\omega(M)$. This provides an extension of previous results
by Bao and Ratiu \cite{bao1997maximal} and El Hadrami \cite{hadrami1996poisson}.
They also prove in \cite{mousavi2026} that after an appropriate completion of
algebras an analogue of the  Schur-Horn-Kostant convexity theorem for $ {\rm Symp}_\omega(M, T)\hookrightarrow  {\rm Symp}_\omega(M)$
holds (relaxing on the way some hypotheses of \cite{bloch1993schur}), as well as describe the set of extreme points as the orbit of the Weyl semi-group through the 
``spectrum" $\Lambda$ of a function $f\in L^p(M)$.

\begin{remark} 
    It is interesting to compare that approach 
    with the setting of the Madelung transform (assuming for simplicity the $C^\infty$ rather than  Sobolev $H^s$ settings in both cases).
    For a {\it toric manifold} $M^{2n}$ elements of the group $ {\rm Symp}_\omega(M)$ corresponds to (generating) {\it functions of $2n$ variables}, while the {\it maximal torus} $ {\rm Symp}_\omega(M, T)$ of equivariant symplectomorphisms is described by {\it functions of $n$ variables}. 
This is similar to the action of $U(n+1)$ on generic orbits in $\mathfrak u(n+1)^*$ of dimension $\sim n^2$
with maximal torus of dimension $\sim n$. On the other hand, in the general {\it Madelung setting} for a manifold $M^n$, the group $\DC$ acts on wave functions, whose functional dimension corresponds to a single
complex-valued function (i.e., {\it two real-valued functions}) of $n$ variables on $M$, while the 
maximal torus $\mathbb T^\infty$ corresponds to {\it one real-valued function} of $n$ variables.
This is similar to the action of $U(n+1)$ on the smallest nontrivial coadjoint orbits, namely, orbits  of matrices of rank 1. Such orbits 
are symplectomorphic to $\mathbb{CP}^n$, while the maximal torus, being of real dimension $n$, corresponds to half of the orbit's dimension.

Thus the infinite-dimensional settings of the convexity results for ${\rm Symp}(M)$ in \cite{bloch1993schur, bao1997maximal, hadrami1996poisson, mousavi2026} on the one hand and for $\DC$-group and Madelung transform in this paper on the other hand are complementary to each other as describing analogues for  actions on, respectively,  the biggest and smallest coadjoint orbits in the unitary setting.
\end{remark}


\appendix
\section{Remarks on the Marsden--Weinstein structures}\label{sect:MW}
Here we relate the above constructions for prequantization of coadjoint orbits of the $\DC$ group to  prequantization of the Marsden--Weinstein symplectic structure, discussed in \cite{ChIs2025}.
Let $\gamma$ be an oriented codimension 2 compact submanifold (also called a membrane or a knot) in a manifold $M$ equipped with a volume form~$\mu$.  Recall that the {\it Marsden--Weinstein symplectic structure} on the space of membranes $\mathcal{M}$ is defined by the formula
$$
    \Omega^{MW}_\gamma(v,w) := \int_\gamma i_vi_w\,\mu
$$
where two tangent vectors $v,w$  are  regarded as a pair of variations of the membrane $\gamma\subset M$, i.e., two vector fields in $M$ defined only on $\gamma$.
If the manifold $M$ is of dimension $n$, the submanifold $\gamma$ is of dimension $n-2$, over which one integrates the $(n-2)$-form $i_vi_w\,\mu$ defined on $\gamma$.

On the other hand, one can regard $\gamma$ as a linear functional (a current, in the sense of distributions) on the space $\mathfrak{X}_\mu(M)$ of (null-homologous) divergence free vector fields in $M$, where the pairing is the flux through $\gamma$. Namely, given $V\in \mathfrak{X}_\mu(M)$ set $\langle \gamma, V\rangle = {\rm Flux}\, V|_\gamma $. 

\begin{proposition}[\cite{MW1983, Brylinski1993}]\label{prop:MWvsKK}
The  Marsden--Weinstein symplectic structure on $\mathcal{M}$ coincides with the Kirillov--Kostant symplectic structure on the coadjoint orbit $\mathcal O_\gamma$ in the natural completion $\overline{\mathfrak{X}^*_\mu}(M)$ of the dual space $\mathfrak{X}^*_\mu(M)$.
\end{proposition}

Recall that null-homologous vector fields $V\in \mathfrak{X}_\mu(M)$ 
are defined by the condition that $i_V\mu$ is not only closed (as for all divergence-free
fields), but an exact $(n-1)$-form. Then  the smooth dual space  $\mathfrak{X}^*_\mu(M)$ 
is naturally identified with the quotient 
$\Omega^1(M)/Z^1(M)$ of all 1-forms on $M$ modulo closed ones, 
or with the space of exact 2-forms on $M$ by taking differential: 
$$
\mathfrak{X}^*_\mu(M):=\Omega^1(M)/Z^1(M)\simeq \mathrm d\Omega^1(M)\,.
$$
 
Consider singular elements in the completion $ \overline{\mathfrak{X}^*_\mu}(M)\supset \mathfrak{X}^*_\mu(M)$ represented by $\delta$-type 2-forms (``currents") $\delta_\gamma$, supported on codimension 2 oriented submanifolds $\gamma\subset M$. Exactness of the current $\delta_\gamma$ corresponds to $\gamma$ being a boundary, $\gamma=\partial\Gamma$. 

\medskip

Noncritical zeros of (complex-valued) wave functions $\psi$ on $M$ are a natural  source of codimension 2 submanifolds, and it is discussed in detail in \cite{ChIs2025}. 
Recall that according to Theorem~\ref{thm:sing-alpha} for  a wave function $\psi \in \WFgamma $ with regular zeros on $\gamma\subset M$
 the 1-form $\alpha=\frac {\rm i}{2} \left(\frac{d\bar\psi}{\bar\psi}-\frac{d\psi}{\psi} \right)$ satisfies
  $\rm d \alpha=2\pi\delta_\gamma\in \overline{\mathfrak{X}^*_\mu}(M)\,.$
Note that if  $\alpha$ were a smooth 1-form on $M$, being also  closed and considered modulo closed forms, its coset $[\alpha]$  would vanish in $ {\mathfrak{X}^*_\mu}(M)$. 
However, $\alpha$ has a pole-like singularity on
$\gamma\subset M$, and hence it can represent a nontrivial element in   $ \overline{\mathfrak{X}^*_\mu}(M)\,.$

\begin{remark}\label{rem:comparison}
Since the Marsden--Weinstein  symplectic structure is defined on the orbit $\mathcal O_\gamma\subset {\mathfrak{X}^*_\mu}(M)$ one can pull back this symplectic form to the space of wave functions.
The corresponding Liouville 1-form is $\alpha$, 
while the prequantum $S^1$-bundle is defined by phases of (certain equivalence classes of) the wave functions, see \cite{ChIs2025} for details. 

Here is a comparison of this ``inversely engineered'' symplectic structure on wave functions with the standard one on $C^\infty(M,\mathbb C)$ considered above for the orbits in~$\mathfrak{dc}^*$.  
\end{remark}

Consider the group $\mathrm{Diff}_\mu(M)$ of (null-homologous) volume-preserving diffeomorphisms of $M$ as a subgroup
of the semi-direct product group $\DC = \mathrm{Diff}(M)\ltimes C^\infty(M, S^1)$:
$$
\mathrm{Diff}_\mu(M) \times \{0\}\subset \mathrm{Diff}(M)\ltimes C^\infty(M, S^1) \,.
$$
Then its Lie algebra  $\mathfrak{X}_\mu(M)\oplus \{0\}\subset \mathfrak{X}(M)\ltimes C^\infty(M,\mathbb{R})$ is  also regarded as the corresponding Lie subalgebra, $\mathfrak{X}_\mu(M)\subset \mathfrak{dc}$. Hence for the dual spaces there is a natural projection $pr: \mathfrak{dc}^*\to \mathfrak{X}^*_\mu(M)$.
The smooth dual space $\mathfrak{dc}^*$ consists of pairs $(\alpha\otimes\varrho, \varrho)$ with
volume forms $\varrho$ and 1-forms  $\alpha$, while the smooth dual space $\mathfrak{X}^*_\mu(M)\simeq \Omega^1(M)/Z^1(M)$ consists of the cosets of 1-forms $[\alpha]$.

\begin{proposition}\label{prop:projLiePoisson}
The projection map $\mathfrak{dc}^*\to \mathfrak{X}^*_\mu(M)$ in  coordinates $(\alpha\otimes\varrho, \varrho)$
corresponds 
$$
(\alpha\otimes\varrho, \varrho)\mapsto \beta :=(\varrho/\mu) \alpha \mapsto [\beta]\in \Omega^1(M)/Z^1(M)\,.
$$ 
This maps takes the Lie-Poisson structure on the dual $\mathfrak{dc}^*$ to the Lie-Poisson structure on the dual $\mathfrak{X}^*_\mu(M)$. The corresponding symplectic reduction by the $\mathrm{Diff}_\mu(M) $-action takes  coadjoint orbits  $\mathcal O_{(\alpha\otimes\varrho,\varrho)}\subset \mathfrak{dc}^*$ to coadjoint orbits $\mathcal O_{[\beta]}\subset \mathfrak{X}^*_\mu(M)$.
\end{proposition}

\begin{proof}
The proof of the projection map is a direct computation based on the  explicit expression of the coadjoint action in $\mathfrak{dc}^*$.
More generally, for any subgroup $N\subset G$, the action of $N$ on the dual $\mathfrak g^*$ is Hamiltonian for the Lie-Poisson structure, while the projection of the duals $\mathfrak n^*\to\mathfrak g^*$ is the corresponding momentum map.
It sends the Lie-Poisson structure on $\mathfrak g^*$ to that on $\mathfrak n^*$,
while it is a symplectic reduction of the corresponding coadjoint orbits. Applying this to the subgroup 
$\mathrm{Diff}_\mu(M) \times \{0\}\subset \mathrm{Diff}(M)\ltimes C^\infty(M, \mathbb{R})$
    implies the result.
\end{proof}

\begin{remark}\label{rem:smooth}
Note that the resulting 1-form $\beta :=(\varrho/\mu) \alpha $ is smooth, since so is $\alpha\otimes\varrho$, while $\mu$ is a non-vanishing volume form. Thus as the result of this reduction one naturally obtains the symplectic structure on 
coadjoint orbits $\mathcal O_{[\beta]}\subset \mathfrak{X}^*_\mu(M)$, rather than on the coadjoint orbits $\mathcal O_\gamma\subset \overline{\mathfrak{X}^*_\mu}(M)$ of singular elements $\delta_\gamma$ 
as needed for the Marsden--Weinstein structure.
\end{remark}



\bibliographystyle{amsplainnat}
\bibliography{bibliography}

\end{document}